\documentclass[]{article}
\usepackage{amsmath}
\usepackage{amssymb}
\usepackage{amsthm}
\usepackage{hyperref}
\usepackage[backend=biber, sorting=nyt, maxnames=99 ]{biblatex} 
\usepackage{authblk}
\bibliography{Phase_Transition_for_Potentials_of_High-Dimensional_Wells_with_Volume_Conservation.bib}
\title{Phase Transition for Potentials of High-Dimensional Wells with a Mass-Type Constraint }
\author[a]{Xingyu Wang}
\author[a]{Yaguang Wang}
\affil[a]{School of Mathematical Sciences, Shanghai Jiao Tong University, Shanghai, China}
\numberwithin {equation} {section}
\newtheorem{theorem}{Theorem}
\newtheorem{lemma}{Lemma}[section] 
\newtheorem{definition}{Definition}
\newtheorem{proposition}{Proposition}[section] 
\theoremstyle{remark} 
\newtheorem{remark}{Remark}[section]
\pdfoutput=1
\newtheorem{claim}{Claim} 
\newtheorem*{claim*}{Claim}
\begin{document}

\maketitle
\begin{abstract}
 Inspired by Lin-Pan-Wang (Comm. Pure Appl. Math., 65(6): 833-888, 2012), we continue to study the corresponding time-independent case of the Keller-Rubinstein-Sternberg problem. To be precise, we explore the asymptotic behavior of minimizers as $\varepsilon\to0$, for the functional
 $$\mathbf{E}_\varepsilon(u)= 
\int_{\Omega}\left(|\nabla u|^{2}+\frac{1}{\varepsilon^{2}} F(u)\right) d x$$under a mass-type constraint $\int_{\Omega}\rho(u)\, dx=m$, where $\rho:\mathbb{R}^k \to \mathbb{R}\in Lip(\mathbb{R}^k)$ is specialized as a density function with $m$ representing a fixed total mass. The potential function $F$ vanishes on two disjoint, compact, connected, smooth Riemannian submanifolds $N^{\pm}\subset\mathbb{R}^k$. We analyze the expansion of $\mathbf{E}_\varepsilon(u_\varepsilon)$ for various density functions $\rho$, identifying the leading-order term in the asymptotic expansion, which depends on the geometry of the domain and the energy of minimal connecting orbits between $N^+$ and $N^-$. Furthermore, we estimate the higher-order term under different geometric assumptions and characterize the convergence $u_{\varepsilon_i}\to v
$ in the ${L}^1$ sense.
\end{abstract}
\newpage
\tableofcontents
\newpage
\section{Introduction}
\subsection{Background and related results}
The theory of singular perturbation for scalar-valued phase transition problems has aroused great interest both in analysis and computations, owing to its important applications to many subjects in sciences. We consider a fluid system with free energy $F(u)$ per unit mass within a smooth bounded region $\Omega\subset\mathbb{R}^n$, where $u(x)$, for $x\in\Omega$, represents the density distribution of the fluid. The classical energy model is given by the Cahn-Hilliard energy functional \cite{gurtin1985theory}
\begin{equation}
\int_{\Omega}\left(|\nabla u|^{2}+\frac{1}{\varepsilon^{2}}F(u)\right) d x, \label{1.1.1}
\end{equation}
where $F:\mathbb{R}\to\mathbb{R}_+$ is a double-well potential function vanishing at $\pm1$. The term $|\nabla u|^2$ represents the interfacial energy (see \cite{gurtin1985theory}). If the fluid in $\Omega$ has a fixed mass $m$, the density distributions $u$ must satisfy the mass constraint, 
\begin{equation}
 \int_{\Omega}u\, dx=m.\label{1.1.2}
\end{equation}
The asymptotic behavior of minimizers $u_\varepsilon$ of the above Cahn-Hilliard energy functional under the mass constraint (\ref{1.1.2}), as $\varepsilon\to0$, was first studied by Modica and Mortola \cite{modica1977limite}, Modica \cite{modica1987gradient} and later by Luckhaus and Modica \cite{luckhaus1989gibbs}. They showed that the separation region between the two stable phases has thickness $O(\varepsilon)$, and the phase transition converges to a minimal hypersurface. Subsequent work by Kohn and Sternberg \cite{kohn1989local, sternberg1988effect} as well as Modica and Mortola \cite{modica1977limite} extended their results within the framework of De Giorgi’s $\Gamma$-convergence. Later studies have explored a broader class of potential functions $F$, revealing various behaviors under different constraints. We now give a review for different potentials $F$:

For potentials of one-dimensional wells $F:\mathbb{R}^k\to \mathbb{R}_+$, vanishing at two points $\{p, q\}\subset\mathbb{R}^k$. For $k=1$, Leoni, Dal Maso, and Fonseca (\cite{leoni2016second, dal2015second}) provided a second-order asymptotic expansion by $\Gamma$-convergence of (\ref{1.1.1}) under the mass constraint $\int_{\Omega}u\, dx=m$. For a rigorous analysis of (\ref{1.1.1}) with Dirichlet conditions, see the works of Burger and Fraekel \cite{berger1970asymptotic}, Caginalp \cite{caginalp1986analysis}, Howes \cite{howes1983perturbed}. Specifically, when $n=1$ and $\Omega\subset\mathbb{R}$ represents an interval, there are several interesting studies; see, for example, \cite{anzellotti1993asymptotic, bellettini2015gamma, carr1984structured}. For $k\geq2$, Fonseca and Tartar \cite{fonseca1989gradient} provided some results.

For potentials of two-dimensional wells $F:\mathbb{R}^2\to \mathbb{R}_+$, vanishing on two smooth and simple closed curves, Nelly-Shafrir \cite{ andre2011minimization, andre2014minimization} and Sternberg \cite{sternberg1991vector} gave the asymptotic analysis for (\ref{1.1.1}) under the mass-type constraint $\int_\Omega|u|dx=m$.

For potentials of high-dimensional wells $F:\mathbb{R}^k\to \mathbb{R}_+$, vanishing on two disjoint, compact, connected, submanifolds $N^{\pm}$ in $\mathbb{R}^k$. Only Lin-Pan-Wang in \cite{lin2012phase}, rigorously justified the formal asymptotic analysis under a suitable Dirichlet boundary condition. To be precise, they established the interface limit of (\ref{1.1.1}), derived a non-standard boundary condition (holds trivially in the case of the scalar Allen-Cahn equation), which essentially generalizes previous results to vectorial cases. Lin-Wang in \cite{lin2023isotropic} studied the phase transition between the isotropic and nematic states of the liquid crystal based on the Ericksen model.

To model certain chemical reaction processes, Rubinstein, Sternberg, and Keller \cite{ rubinstein1989reaction,rubinstein1989fast} introduced a vector-valued system of fast reaction and slow diffusion:
\begin{equation}
 \partial u=\varepsilon\Delta u-\varepsilon^{-1}F_u(u)\quad\text{ in } \Omega, \label{allen}
\end{equation}
where the vector $u$ belongs to $\mathbb{R}^k$, $F_u(u)$ determines the rate of reaction, and the term $\Delta u$ describes the effect of diffusion. The parameter $\varepsilon$ quantifies the relative importance of reaction versus diffusion processes. This system of equations has a variational structure, offering valuable insight into its dynamics. Specifically, the equation (\ref{allen}) can be regarded as the gradient flow of the functional, 
\begin{equation}
\int_{\Omega}\left(\varepsilon|\nabla u|^{2}+\frac{1}{\varepsilon^{}} F(u)\right) d x.\label{functional}
\end{equation}According to the law of mass action during the chemical reaction process, the potential function $F: \mathbb{R}^k\to \mathbb{R}_+$ can vanish on two disjoint submanifolds in $\mathbb{R}^k$. In such cases, a front develops in $\Omega$ and moves by its mean curvature (see, \cite{ rubinstein1989reaction,rubinstein1989fast}).
Extensive research has focused on a rigorous analysis for the scalar case $k=1$ (see, for example, \cite{bronsard1991motion, chen1992generation, dong2021gradient, evans1992phase}). For $k\geq2$, Fei-Wang-Zhang-Zhang \cite{fei2015dynamics, fei2018isotropic} investigated this problem within the liquid crystal framework (isotropic-nematic) under periodic boundary conditions. Furthermore, Fei-Lin-Wang-Zhang \cite{fei2023matrix} solved the Keller-Rubinstein-Sternberg problem in the $O(n)$ setting with periodic boundary conditions. However, generalization to vector-valued systems remains an open problem, particularly when boundary effects are included. For a rigorous analysis for the scalar case $k=1$ with boundary effects, refer to \cite{abels2019convergence, abels2022convergence, moser2023convergence}.


In summary, considerations from both dynamic and variational perspectives drive the study of the time-independent case of the Keller-Rubinstein-Sternberg problem. In this work, we investigate the asymptotic behavior, as $\varepsilon\to0$, of the minimizers $u_\varepsilon$ of (\ref{functional}) under the constraint
$$\int_{\Omega}\rho(u_\varepsilon)\, dx=m. $$
Here, $\rho:\mathbb{R}^k \to \mathbb{R}\in Lip(\mathbb{R}^k)$ is a suitably specialized density function, and $m$ denotes a fixed constant (see (\ref{mass}) below). 
\subsection{The description of the problem \label{subsetion1.2}}
Let us describe our problem. For the sake of generality, we will follow the settings described in \cite{lin2012phase}. For $k \geq 1$, let
$$
N=N^{+} \cup N^{-} \subset \mathbb{R}^{k}
$$
be the union of two disjoint, compact, connected, smooth Riemannian submanifolds $N^{\pm} \subset\mathbb{R}^{k}$ without boundaries. 

Define
\begin{equation}
 d_{N}=\inf \left\{\left|p^{+}-p^{-}\right|: p^{ \pm} \in N^{ \pm}\right\} \quad\left(>2 \delta_{N}\right)\label{d_N define}
\end{equation}
as the Euclidean distance between $N^{+}$and $N^{-}$. Let
\begin{align}
& M^{+}=\left\{p^{+} \in N^{+}: \exists p^{-} \in N^{-} \text { s.t. }\left|p^{+}-p^{-}\right|=d_{N}\right\} \\
& M^{-}=\left\{q^{-} \in N^{-}: \exists q^{+} \in N^{+} \text { s.t. }\left|q^{+}-q^{-}\right|=d_{N}\right\}\label{1.1.4}
\end{align}
be the pair of sets of points in $N^{ \pm}$ achieving $d_{N}$. For $\delta>0$, let
$$
N_{\delta} \equiv\left\{p \in \mathbb{R}^{k}: d(p, N) \equiv \inf _{y \in N}|p-y| \leq \delta\right\}
$$
denote the $\delta$-neighborhood of $N$. It is well-known there exists $\delta_{N}>0$ depending only on $N$, such that
\begin{align*}
 &(1)\text{ the nearest-point projection map } \Pi: N_{\delta_{N}} \to N \text{ is smooth, and }\\
 &(2) \text{ } d^{2}(p, N)=|p-\Pi(p)|^{2} \in C^{\infty}\left(N_{\delta_{N}}\right).
\end{align*}

Throughout, we assume $\Omega\subset\mathbb{R}^n,n\geq1$ is a bounded and smooth region. Without loss of generality, we also assume $\Omega$ has unit mass, that is
\begin{equation*}
\mu(\Omega)=1, 
\end{equation*}
where $\mu$ denotes the $n$-dimensional Lebesgue measure.

Now, let us introduce the mass-type constraint. Let the density function $\rho:\mathbb{R}^k \to \mathbb{R}\in Lip(\mathbb{R}^k)$ satisfy the following conditions:
\begin{equation}
 \rho(N^-)=[a, b], \quad \rho(N^+)=[c, d], \label{classify}
\end{equation}
where $a, b, c, d\in\mathbb{R}$. We allow $a=b$ or $c=d$, but require $b<c$. For $u\in H^1$, a natural extension of the classical mass constraint (\ref{1.1.2}) is
\begin{equation}
 \int_{\Omega}\rho(u)dx=m, \label{mass}
\end{equation}
where $m\in(b, c)$, to enforce the desired phase transition. In fact, when $a\leq m\leq b$ or $c\leq m\leq d$, the energy (\ref{functional}) achieves its minimum value of zero under a constant mapping. For $m< a$ or $m>d$, the energy (\ref{functional}) diverges to infinity as $\varepsilon\to0$, since $\mu\{u\notin N\}>0$ holds for all $ u\in H^1$ satisfying (\ref{mass}).

We classify the mass constraint based on the ranges of value of the density function defined over the phases $N^+$ and $ N^-$.
\begin{definition}
 For $\rho:\mathbb{R}^k \to \mathbb{R}\in Lip(\mathbb{R}^k)$, satisfying (\ref{classify}), the constraint (\ref{mass}) is classified as follows: 
 \begin{equation}
 \begin{aligned}
 \text{ type I mass constraint }\quad &\text{ if }\quad a<b \text{ or } c<d, \\
 \text{ type II mass constraint }\quad\text\quad&\text{ if } \quad a=b \text{ and } c=d. 
 \end{aligned}
 \end{equation}
\end{definition}
For the widely studied density function $ \rho$, the following classifications apply: 
\begin{itemize}
  \item If $\rho(p)=|p|$ for $p\in\mathbb{R}^2$, $N^+$ and $N^-$ are two disjoint, smooth, and simple closed curves in $\mathbb{R}^2$. The constraint (\ref{mass}), classified as type I mass constraint, corresponds to the constraint studied in \cite{andre2011minimization, andre2014minimization, sternberg1988effect}. 
  \item If $\rho(p)=p$ for $p\in\mathbb{R}$, the constraint (\ref{mass}), classified as the type II mass constraint, aligns with the classical mass constraint (\ref{1.1.2}).
\end{itemize}

In this study, we analyze both type I and type II mass constraints as follows.

\begin{itemize}
\item[($\mathbf{i}$)] Type I mass constraint: We assume the density function $\rho:\mathbb{R}^k \to \mathbb{R}\in Lip(\mathbb{R}^k)$ satisfies the following properties:
\begin{equation}
 \rho(N^-)=\{0\}, \quad \rho(N^+)=[m_1, m_2], \quad m_2>m_1>0.\label{mass constraint}
\end{equation}
This imposes the mass constraint 
\begin{equation}
 \int_{\Omega}\rho(u)dx=m, \quad m\in(0, m_1).\label{W constraint}
\end{equation}
 The argument in this paper for density function $\rho$ in (\ref{mass constraint}) is also applicable for the case where $\rho(N^-)$ is an interval disjointing with $[m_1, m_2]$. For simplicity of description, we only consider $\rho(N^-)=\{0\}$.

\item[($\mathbf{ii}$)] Type II mass constraint: We define the density function using a modified distance function, 
\begin{equation}
 \widetilde{\rho}(p):=d_0(p, N)=\begin{cases}
 d(p, N^-)- \frac{d_N}{2}& \quad\text{ if } d(p, N^-)\leq \frac{d_N}{2}, \\
 \frac{d_N}{2}-d(p, N^+)& \quad\text{ if } d(p, N^+)\leq \frac{d_N}{2}, \\
 0&\quad\text{ if } d(p, N)\geq \frac{d_N}{2}, 
 \end{cases}\label{d_0 def}
\end{equation}
where $d_N$ is the minimum distance between $N^+$ and $ N^-$ (\ref{d_N define}), and $\widetilde{\rho}$ belongs $W^{1, \infty}(\mathbb{R}^k, \mathbb{R}_+)$, as proved in Section \ref{section 5} . From this definition, it follows that
\begin{equation}
\widetilde{\rho}(N^-)=\{-\frac{d_N}{2}\}, \quad \widetilde{\rho}(N^+)=\{\frac{d_N}{2}\}. \label{mass constraint one}
\end{equation}
This imposes the mass constraint 
\begin{equation}
 \int_{\Omega}\widetilde{\rho}(u)dx=m, \quad m\in(-\frac{d_N}{2}, \frac{d_N}{2}).\label{W constraint 1}
\end{equation}

\end{itemize}

It is important to emphasize that, in many cases, even for high-dimensional potential wells, the mass constraint problem often reduces to effectively one-dimensional problems. For the type I mass constraint, this reduction is not possible due to the non-uniform density of the phase $N^+$. In contrast, for the type II mass constraint, it is essentially a one-dimensional problem by the P\'olya-Szeg{\H o} inequality (a detailed introduction to this will be provided in Section \ref{section 5} ). Although our study for this scenario is limited to Section \ref{section 5}, where it is presented as an application of the results derived in \cite{leoni2016second}.

To state our results, we first recall the concept of the isoperimetric profile of a domain. Refer to \cite{ros2001isoperimetric}; the isoperimetric profile $I_{\Omega}:(0, 1) \to \mathbb{R}$ of $\Omega$ is defined by
\begin{equation}
I_{\Omega}(t)=\min \left\{\operatorname{Per}_{\Omega} A: A \subset \Omega \text { s.t. } \chi_{A} \in BV(\Omega) \text { and } \mu(A)=t\right\}, \label{minimizers}
\end{equation}
where $BV(\Omega)$ denotes the space of functions of bounded variation, and $\operatorname{Per}_{\Omega} A$ is the perimeter of $A$ in $\Omega$ (see Section \ref{section 2.1}). This minimum of (\ref{minimizers}) can be achieved by the lower semi-continuity of the $BV$ norm (Lemma \ref{lemma infcontinuous}) and the compactness of $BV$ in $L^1$ (Lemma \ref{lemma compact}).
The properties of minimizers of (\ref{minimizers}) have been extensively studied by Gr\"uter \cite{gruter1987boundary} (also see \cite{howes1983perturbed, sternberg1999connectivity}), who showed that when $\Omega$ is bounded and of class $C^{2}$, minimizers of (\ref{minimizers}) exist, have constant generalized mean curvature $\kappa_{E}$, and intersect the boundary of $\Omega$ orthogonally. In addition, their singular set is empty if $n \leq 7$, and has a Hausdorff dimension of at most $n-8$ if $n \geq 8$.

In the following part of the article, we focus on the type I mass constraint (\ref{W constraint}). For the type II mass constraint (\ref{W constraint 1}), we present all results in Section \ref{section 5}. 

\subsection{Main results }

Assume
\begin{equation}
 F(p)=f\left(d^{2}(p, N)\right), \label{1.1.401}
\end{equation}
 where $F$ attains its minimal value zero on $N$, and $f \in C^{\infty}\left(\mathbb{R}_{+}, \mathbb{R}_{+}\right)$ satisfies the properties: there exists $c_{1}, c_{2}, c_{3}, c_4, t_0>0$ such that
\begin{align}
 \begin{cases}
 c_{1} t \leq f(t) \leq c_{2} t & \text { if } 0 \leq t \leq \delta_{N}^{2}, \\ 
 f(t) \geq c_{3} & \text { if } t \geq \delta_ {N}^{2}, \\
 f(t)\geq c_4t^{\frac{1}{2}} & \text { if } t\geq t_0.\end{cases}
 \tag{$H$}\label{H1}
\end{align}

For $u\in H^1(\Omega)$, we define
\begin{equation*}
\mathbf{E}_{\varepsilon}(u):= 
\int_{\Omega}\left(|\nabla u|^{2}+\frac{1}{\varepsilon^{2}} F(u)\right) d x, 
\end{equation*}
and
\begin{equation}
\mathbf{E}_{\varepsilon}:=\min_{u\in H^1(\Omega)} \left\{
\mathbf{E}_{\varepsilon}(u):\int_{\Omega}\rho(u)dx=m
\right\}, \label{1.9}
\end{equation}
where $m$ is a fixed constant satisfying $0<m<m_1$. The existence of the minimizer $u_\varepsilon$ of (\ref{1.9}) can be established using the direct method of the calculus of variations (by the weak compactness in $H^1(\Omega)$, as ensured by the properties (\ref{H1})).

Let $ c_0^F$ be the energy of the minimal connecting orbits (see Section \ref{subsection 2.2} below for details) between $N^+$ and $N^-$, defined by 
\begin{equation}
c_0^F:=\inf \left\{c^{F}\left(p^{+}, p^{-}\right): p^{ \pm} \in N^{ \pm}\right\} \label{1.2.11}
\end{equation}
where
\begin{equation}
 \begin{aligned}
& c^{F}\left(p^{+}, p^{-}\right):=\\
& \quad \inf \left\{\int_{\mathbb{R}}\left(\left|\xi^{\prime}(t)\right|^{2}+F(\xi)\right) d t: \xi \in H^{1}\left(\mathbb{R}, \mathbb{R}^{k}\right), \xi(\pm \infty)=p^{ \pm}\right\}.\label{1.2.12}
\end{aligned}
\end{equation}
Let $\sigma_m\in[\frac{m}{m_2}, \frac{m}{m_1}]$ such that
\begin{equation}
 I_\Omega(\sigma_m)=\min_{[\frac{m}{m_2}, \frac{m}{m_1}]}I_\Omega(\cdot).\label{1.7}
\end{equation}

The existence of such $\sigma_m$ is guaranteed by the continuity of $I_\Omega$ over the interval $[\frac{m}{m_2}, \frac{m}{m_1}]$.

Now, we state our first main result.
\begin{theorem}
Assume
\begin{equation}
 F(p)=f\left(d^{2}(p, N)\right), 
\end{equation}
$\mathbf{E}_{\varepsilon}$ is defined in (\ref{1.9}). Then, we have
\begin{equation}
\lim_{\varepsilon\to0}\varepsilon\mathbf{E}_\varepsilon = c_0^F I_{\Omega}({\sigma_m}).
\end{equation}

\label{theorem 1}
\end{theorem}
\begin{remark}
\text{}{}

(1) Theorem \ref{theorem 1} confirms Remark 1.2 (4)(7)(8) in \cite{lin2012phase}. Remarks hypothesize that the limiting problems are isoperimetric problems or partition problems under mass-type constraints, and the maps describing sharp transitions between various regions behave essentially like in a one-dimensional problem.

(2) Theorem \ref{theorem 1} was previously proven by \cite{fonseca1989gradient, kohn1989local, leoni2016second,sternberg1991vector} when $N=$ $\{p, q\} \subset \mathbb{R}^{k}$ is a set of two points under the mass constraint $ \int_{\Omega} u d x=m\in\mathbb{R}^k$, by \cite{sternberg1988effect} when $N \subset \mathbb{R}^{2}$ is the union of two concentric simple closed curves under the mass constraint $ \int_{\Omega} |u| d x=m$ .

\label{remark}
\end{remark}

Theorem \ref{theorem 1} implies
\begin{equation}
 \mathbf{E}_\varepsilon=\frac{c_0^FI_{\Omega}({\sigma_m})}{\varepsilon}+\mathbf{R}_\varepsilon, \quad \mathbf{R}_\varepsilon=\frac{o(1)}{\varepsilon}.
\end{equation}
As in \cite{andre2011minimization, andre2014minimization, lin2012phase}, we care about how to characterize the remainder $\mathbf{R}_\varepsilon$ as $\varepsilon\to0$. To characterize $\mathbf{R}_\varepsilon$, we require additional properties on the dimension of the region $n$ and the isoperimetric profile $I_\Omega$. For $n$, we restrict $n\leq7$ to ensure the regularity of minimizers of (\ref{minimizers}). For the isoperimetric profile $I_\Omega$, we assume there exists ${\sigma_m}\in (\frac{m}{m_2}, \frac{m}{m_1})$ such for some small $\eta>0$, 
\begin{equation}
 I_\Omega({\sigma_m})\leq\min_{[\frac{m}{m_2}-\eta, \frac{m}{m_1}+\eta]}I_\Omega(\cdot).\label{G1} \tag{$G$}
\end{equation}
Using the condition (\ref{G1}), we know that there exists $\Omega^+\subset\Omega$, $\mu(\Omega^+)={\sigma_m}\in(\frac{m}{m_1}, \frac{m}{m_2})$, and 
\begin{equation}
 \Gamma:=\overline{\partial \Omega^+\cup\Omega}\text{ and } \Omega^-:=\Omega\backslash\Omega^+\label{gammer}
\end{equation}
such that $\operatorname{Per}_\Omega \Omega^+=I_{\Omega}({\sigma_m})=H^{n-1}(\Gamma)$. In fact, if $I_\Omega$ does not reach the minimum $I_\Omega({\sigma_m})$ over the bounds $\{\frac{m}{m_2}, \frac{m}{m_1}\}$, condition (\ref{G1}) is satisfied. In particular, condition (\ref{G1}) corresponds to the property $NC$ for the region $\Omega$ as discussed in \cite{andre2011minimization}.

Now, we state our second theorem.
\begin{theorem}
Assume $n\leq7$ and $ I_\Omega$ satisfies condition (\ref{G1}). Assume
\begin{equation}
 F(p)=f\left(d^{2}(p, N)\right), 
\end{equation}
$\mathbf{E}_{\varepsilon}$ is defined in (\ref{1.9}). Then, for small $\varepsilon$,

\begin{equation}
 \mathbf{E}_\varepsilon = \frac{c_0^F I_{\Omega}({\sigma_m})}{\varepsilon}+O(1), 
\end{equation}
where $c_0^F$ and $I_{\Omega}({\sigma_m})$ are defined same as in Theorem \ref{theorem 1}.
\label{theorem 2}
\end{theorem}

\begin{remark}
\text{{}}

(1) Although the estimate $O(1)$ of the remainder $\mathbf{R}_\varepsilon$ is not optimal, it suffices to establish Theorem \ref{Theorem 3} as stated below. For a complete characterization for the remainder $\mathbf{R}_\varepsilon$ under additional constraints, refer to (\cite{lin2012phase}, Theorem 1.3) and (\cite{andre2014minimization}, Theorem 2). 

(2) Theorem \ref{theorem 2} was previously proven by Andr{\'e} and Shafrir \cite{andre2011minimization} for $N^\pm$, the disjoint union of two simple closed curves in $\mathbb{R}^{2}$ under the mass constraint $ \int_\Omega|u|dx=c.$
\end{remark}
Theorems \ref{theorem 1} and \ref{theorem 2} provide a characterization for $\mathbf{E}_\varepsilon$.
The next goal is to characterize the convergence for $u_{\varepsilon_i}\to v
$ for any $L^1$-convergent subsequence of $\{u_\varepsilon\}$. In contrast to the one-dimensional phase, $u$ cannot be determined by condition $ F (u) =0$. For $\{u_\varepsilon\}$, there is no sobolev compactness. This following result relies on the characterization for $\mathbf{E}_\varepsilon$ in Theorem \ref{theorem 2}, and compactness is obtained in $SBV$ spaces, as described in detail in Section \ref{section 4}.
\begin{theorem}
     Under the same assumptions of Theorem \ref{theorem 2}. Let $u_\varepsilon$ be the minimizer of (\ref{1.9}) under $\int_{\Omega}\rho(u)\, dx=m$. Then, there exists $u\in L^1(\Omega, N)$ such that after passing possible subsequences, $u_{\varepsilon_{i}}$ converges to $u$ in $L^1(\Omega, \mathbb{R}^k)$, and u satisfies
    \begin{itemize}
\item[(1)]$u\in H^1(\Omega^\pm, N^\pm)$,
\item[(2)]$\int_\Omega \rho(u)dx=m$,
\label{Theorem 3}
    \end{itemize}
where $\Omega^+$ realizes the minimum $I(\sigma_m)$ and $\Omega^-:=\Omega\backslash\Omega^+$.
   \end{theorem}
\begin{remark}
\leavevmode
Our results are consistent with those under the Dirichlet boundary condition considered in \cite{lin2012phase}. Unfortunately, we do not have information on the one-sided trace of $u$ on $\Gamma$, because the mass-type constraint has unique difficulties with the Dirichlet boundary constraint.
\end{remark}

 The remainder of the paper is organized as follows. In Section \ref{section 2}, we recall the fundamental results that will be employed throughout the work. The main definitions and the theory of minimal connection developed by Lin-Pan-Wang \cite{lin2012phase} are included. Section \ref{section 3} establishes both lower bound and upper bound estimates for Theorem \ref{theorem 1} and Theorem \ref{theorem 2}. In Section \ref{section 4}, we prove Theorem \ref{Theorem 3}. In Section \ref{section 5}, we apply the results from \cite{leoni2016second} to investigate the type II mass constraint.

\section{Preliminaries \label{section 2}}
\subsection{Functions of bounded variation and Federer's co-area\label{section 2.1} formula}

\begin{definition}{($BV(\Omega)$)}
The space of functions of bounded variation, $BV(\Omega)$, consists of all $u \in {L}^1(\Omega)$ for which $\int_{\Omega}|\nabla u|<\infty$, where
\begin{equation*}
 \int_{\Omega}|\nabla u|=|Du|(\Omega):=\sup _{\substack{g \in C_{0}^{1}\left(\Omega, R^{n}\right) \\|g| \le 1}} \int_\Omega u(\nabla \cdot g(x)) d x.
\end{equation*}
Here, $|\nabla u|(\Omega)$ represents the total variation of $u$ in $\Omega$. $B V(\Omega)$ is a Banach space under the norm:
\begin{equation}
 \|u\|_{B V(\Omega)}=\int_{\Omega}|u| d x+\int_{\Omega}|\nabla u|.\label{1.2.1}
\end{equation}

\end{definition}
We now define the perimeter of a set $A$ in $\Omega$:
\begin{definition}
 \begin{equation}
 \begin{aligned}
\operatorname{Per}_{\Omega} A&=\text { perimeter of } A \text { in } \Omega=\int_{\Omega}\left|\nabla \chi_{A}\right|
=\sup _{\substack{g \in C_{0}^{1}\left(\Omega, R^{n}\right) \\|g| \le 1}} \int \chi_{A}(\nabla \cdot g(x)) d x.\label{1.2.2}
\end{aligned}
\end{equation}
\end{definition}
Then, let us introduce two well-known lemmas about the $BV$ space.
\begin{lemma}{(Lower semi-continuity)}\cite{evans2018measure} If $u_{\varepsilon} \to u$ in ${L}^1(\Omega)$, then
\begin{equation}
 \lim _{\varepsilon} \inf \int_{\Omega}\left|\nabla u_{\varepsilon}\right| \geq \int_{\Omega}|\nabla u|.\label{1.2.3}
\end{equation}
\label{lemma infcontinuous}
\end{lemma}
\begin{lemma}{(Compactness of $BV$ in ${L}^1$)}\cite{evans2018measure} Bounded sets in the $BV$ norm are compact in the ${L}^1$ norm.\label{lemma compact}
\end{lemma}
We present Federer's co-area formula for measurable functions, which will be frequently utilized in subsequent sections.
\begin{lemma}
 Assume $f: \Omega \to \mathbb{R}$ is Lipschitz continuous, and suppose $g: \Omega \to \mathbb{R}$ satisfies $g\in {L}^n$, for $t \in \mathbb{R}^1$, we define $f^{-1}(t)=\left\{x \in \Omega: f(x)=t\right\}$, then 
 \begin{equation}
 \int_{\Omega} g |\nabla f |dx=\int_{\mathbb{R}}\left(\int_{f^{-1}(t)} g dH^{n-1}(x)\right) dt.\label{1.2.4}
 \end{equation}
\end{lemma}

\subsection{Minimal connections \label{subsection 2.2}}
In this subsection, we describe the fundamental properties of minimal orbits, which play crucial roles in estimating the upper bounds of $\mathbf{E}_\varepsilon$.

The following function $\overline{F}(\lambda)$ is defined to analyze minimal connections: 
\begin{equation}
\overline{F}(\lambda)= \begin{cases}f\left(\left(\frac{d_{N}}{2}+\lambda\right)^{2}\right) & \text { if } \lambda \leq 0, \\
f\left(\left(\frac{d_{N}}{2}-\lambda\right)^{2}\right)& \text { if } \lambda \geq 0.\end{cases}\label{1.2.6}
\end{equation} 
We now review key properties of the associated, scalar-valued minimal connection problem:
\begin{equation}
c_{0}^{\overline{F}}:=\min \left\{\int_{\mathbb{R}}\left(\left|\zeta^{\prime}(t)\right|^{2}+\overline{F}(\zeta(t))\right) d t: \zeta \in H^{1}(\mathbb{R}), \zeta(\pm \infty)= \pm \frac{d_{N}}{2}\right\}, \label{1.2.7}
\end{equation}
where $\overline{F}: \mathbb{R} \to \mathbb{R}_{+}$ is the even function induced from $F$ which defined by (\ref{1.1.401}). The following properties hold, as derived in \cite{fonseca1989gradient, sternberg1988effect}:
\begin{itemize}
  \item[(1)] It holds that
\begin{equation}
c_{0}^{\overline{F}}=4 \int_{0}^{\frac{d_{N}}{2}} \sqrt{\overline{F}(\lambda)} d \lambda \quad\left(=4 \int_{0}^{\frac{d_{N}}{2}} \sqrt{f\left(\lambda^{2}\right)} d \lambda\right).\label{1.2.8}
\end{equation}

\item[(2)] There exists a minimizer
\begin{equation}
 \alpha \in C^{\infty}\left(\mathbb{R}, \left(-\frac{d_{N}}{2}, \frac{d_{N}}{2}\right)\right)\label{1.2.9}
\end{equation}
of (\ref{1.2.7}) that is odd and strictly monotone increasing, and satisfies
\begin{equation}
 \begin{aligned}
-\alpha^{\prime \prime}(t)+\frac{1}{2} \overline{F}^{\prime}(\alpha(t))=0, \quad t \in \mathbb{R}, \quad \alpha(\pm \infty)= \pm \frac{d_{N}}{2}, \\
c_{0}^{\overline{F}}=\int_{\mathbb{R}}\left(\left|\alpha^{\prime}(t)\right|^{2}+\overline{F}(\alpha(t))\right) d t, \\
\alpha^{\prime}(t)=\sqrt{\overline{F}(\alpha(t))} \quad \forall t \in \mathbb{R}, \\
\left|\alpha^{\prime}(t)\right|+\left|\alpha(t)+\frac{d_{N}}{2}\right| \leq C_{1} e^{C_{2} t} \quad \text { as } t \to-\infty, \\
\left|\alpha^{\prime}(t)\right|+\left|\alpha(t)-\frac{d_{N}}{2}\right| \leq C_{3} e^{-C_{4} t} \quad \text { as } t \to+\infty, \label{1.2.10}
\end{aligned}
\end{equation}
for some $C_{i}>0(1 \leq i \leq 4)$.
\end{itemize}

The following result is established in \cite{lin2012phase}:
\begin{lemma}[\cite{lin2012phase}, Theorem 2.1]
\begin{equation}
 \begin{aligned}
c_0^F=c_{0}^{\overline{F}} & =4 \int_{0}^{\frac{d_{N}}{2}} \sqrt{f\left(\lambda^{2}\right)} d \lambda.\\
c_0^F= & c^{F}\left(p^{+}, p^{-}\right) \quad \text { for } p^{ \pm} \in N^{ \pm} \Longleftrightarrow p^{ \pm} \in M^{ \pm} \text { and }\left|p^{+}-p^{-}\right|=d_{N}, \label{1.2.13}
\end{aligned}
\end{equation}
and the corresponding minimal connecting orbit $\gamma \in H^{1}\left(\mathbb{R}, \mathbb{R}^{k}\right)$ attaining $c_0^F$ is 
\begin{equation}
\gamma(t)=\frac{p^{+}+p^{-}}{2}+\alpha(t) \frac{p^{+}-p^{-}}{\left|p^{+}-p^{-}\right|} \quad t \in \mathbb{R}, \label{1.2.14}
\end{equation}
where $\alpha \in H^{1}(\mathbb{R})$, with $\alpha(\pm \infty)= \pm \frac{d_{N}}{2}$, is a solution to the associated scalar-valued minimal connection problem (\ref{1.2.7}).
\label{minimal connection}
\end{lemma}
Notice that potential functions satisfying (\ref{1.1.401}) are functions of distance to the target manifold. Lemma \ref{minimal connection} tells us that the corresponding minimal connecting orbits are straight lines. 

To facilitate estimating the upper bound of $\mathbf{\mathbf{E}_\varepsilon }$, we introduce a truncation function of $\alpha(t)$ that was used in \cite{lin2012phase}.
\begin{lemma}
Define
\begin{equation}
 \alpha_{2 L}(t)= \begin{cases}\frac{t+2 L}{L} \alpha(-L)+\frac{t+L}{L} \frac{d_{N}}{2} & \text { if }-2 L \leq t \leq-L, \\ \alpha(t)& \text { if }-L \leq t \leq L, \\ \frac{2 L-t}{L} \alpha(L)+\frac{t-L}{L} \frac{d_{N}}{2} & \text { if } L \leq t \leq 2 L, \end{cases}\label{1.2.15}
\end{equation}
where $\alpha \in H^{1}(\mathbb{R})$, with $\alpha(\pm \infty)= \pm \frac{d_{N}}{2}$, is a minimizer of (\ref{1.2.7}). For $\alpha_{2 L}$, we have the following:

(i) $\alpha_{2 L}$ is monotonically increasing, 
$$
-\frac{d_{N}}{2}<\alpha_{2 L}(t)<\frac{d_{N}}{2}, \quad \alpha_{2 L}(\pm 2 L)= \pm \frac{d_{N}}{2}, 
$$
and there exists $C>0$ such that
\begin{equation}
\max _{|t| \leq 2 L}\left(\left|\alpha_{2 L}^{\prime}(t)\right|^{2}+\overline{F}\left(\alpha_{2 L}(t)\right)\right) \leq C. \label{2.1.151}
\end{equation}

(ii) There exists $L_{0}>0$ and $c_{1}, c_{2}>0$ such that for any $L \geq L_{0}$, we have
\begin{equation}
 \begin{aligned}
\int_{-2 L}^{2 L}\left(\left|\alpha_{2 L}^{\prime}(t)\right|^{2}+\overline{F}\left(\alpha_{2 L}(t)\right)\right) d t \leq c_0^F+c_{2} e^{-c_{1} L}.\label{1.2.16}
\end{aligned}
\end{equation}
\label{lemma alpha properties}
\end{lemma}
\begin{proof}
 The above properties of $\alpha_{2 L}$ follow directly from the property (\ref{1.2.10}). For a complete proof, refer to (\cite{lin2012phase}, Proposition A.4).
\end{proof}
Throughout this paper, we denote by $C$ a universal constant that is independent of $\varepsilon$, by $n$ the dimension of $\Omega$, and by $k$ the dimension of the wells $N^\pm$.

\section{Proof of Theorem \ref{theorem 1} and \ref{theorem 2}\label{section 3}}

\subsection{Upper bound estimates of \texorpdfstring{$\mathbf{E}_\varepsilon $}{} \label{section 3.1}}

 This subsection focuses on deriving the upper bounds of $\mathbf{E}_\varepsilon $. To begin, we recall a fundamental result, which will serve as a key tool in the analysis (see, \cite{sternberg1988effect}).
\begin{lemma}
 Let $\Omega$ be an open bounded subset of $\mathbb{R}^{n}$ with a smooth boundary. Let A be an open subset of $\mathbb{R}^{n}$ with $C^{2}$, compact, nonempty boundary such that $H^{n-1}(\partial A \cap \partial \Omega)=0$.
 Define the distance function to $\partial A, d: \Omega \to \mathbb{R}^{}$, by
 \begin{equation*}
 {{d}_{\Gamma}(x)}=\begin{cases}
-\operatorname{dist}(x, \partial A) & x \in \Omega \backslash A, \\
\operatorname{dist}(x, \partial A) & x \in A \cap \Omega.
\end{cases}
 \end{equation*}
Then for some $s>0$, $d$ is a $C^{2}$ function in $\{|{{d}_{\Gamma}(x)}|<s\}$ with
\begin{equation*}
|\nabla d|=1. 
\end{equation*}
Furthermore, as $s\to0$, 
\begin{equation}
 H^{n-1}(\{{{d}_{\Gamma}(x)}=s\})=H^{n-1}(\partial A)+O(s). \label{approxite}
\end{equation}
\label{technical lemma 2}

\end{lemma}
We divide the process into two parts: (1) rough bounds and (2) refined bounds.
\subsubsection{Rough upper bound estimates of \texorpdfstring{$\mathbf{E}_\varepsilon $}{} }
\quad In this subsubsection, we construct comparison maps that satisfy the mass constraint. In construction, we utilize the following Lemma \ref{technical lemma 1}, a classical approximation technique to approximate the minimizer of (\ref{minimizers}). We also employ Lemma \ref{technical lemma 2} to construct the phase transition layer. This argument is classic (see \cite{sternberg1988effect}). 
\begin{lemma}{(Technical lemma)}
 Let $\Omega$ be an open and bounded subset of $\mathbb{R}^{n}$ with a smooth boundary. Let $A \subset \Omega$ be a finite perimeter set in $\Omega$ with $0<|A|<|\Omega|$. Then, there exists a sequence of open sets $\left\{A_{k}\right\}$ satisfying the following five conditions:
\begin{align*}
& 1.\text{ }\partial A_{k} \cap \Omega \in C^{2}, \\
& 2.\text{ }\left|\left(A_{k} \cap \Omega\right) \Delta A\right| \to 0 \quad\text{as} 
 \quad k \to \infty, \\
& 3.\text{ }\operatorname{Per}_{\Omega} A_{k} \to \operatorname{Per}_{\Omega} A \quad\text{as}\quad k \to \infty, \\
&4.\text{ } H^{n-1}\left(\partial A_{k} \cap \partial \Omega\right)=0, \\
&5.\text{ }\left|A_{k} \cap \Omega\right|=|A|.
\end{align*}\label{technical lemma 1}
\end{lemma}
Let us start with the proof of Proposition \ref{upper}:

\begin{proposition}
 Under the same assumptions of Theorem \ref{theorem 1}, we have
 \begin{equation}
\limsup_{\varepsilon \to 0} \varepsilon \mathbf{E}_\varepsilon\leq c_0^F I_{\Omega}({\sigma_m}).
\end{equation}
\label{upper}

\end{proposition}
\begin{proof}
\text{{}}

$Step$ 1. Let $A$ be a set of finite perimeter in $\Omega$ which attains the minimum $I_{\Omega}({\sigma_m})$ of (\ref{minimizers}) at $\sigma_m$. This implies
$$
I_{\Omega}({\sigma_m})=\operatorname{Per}_{\Omega}A.
$$
Define $B=\Omega\backslash A$. For $\forall \delta>0$, by Lemma \ref{technical lemma 1}, there exists an open set $O\subset\mathbb{R}^n$ such that
\begin{equation}
 \begin{aligned}
 & 1.\text{ } \partial O\cap \Omega \in C^{2}, \\
&2. \text{ }\left|H^{(
n-1)}(\partial O)-\operatorname{Per}_{\Omega}A\right|<\delta, \\
&3.\text{ }\left|O \cap \Omega\right|=|A|,\\
&4.\text{ } H^{n-1}\left(\partial O \cap \partial \Omega\right)=0.\end{aligned}
\end{equation}
Let $\Gamma:=\partial O \cap \Omega$, define the signed function $d: \Omega \to \mathbb{R}$ with respect to $\Gamma$ as
$$
{{d}_{\Gamma}(x)}=\left\{\begin{aligned}
d(x, \Gamma) \quad& x \in O\cap\Omega, \\
-d(x, \Gamma) \quad& x \in \Omega\backslash O. 
\end{aligned}\right.
$$
By Lemma \ref{technical lemma 2}, $d_{\Gamma}$ is a $C^2$ function in an open neighborhood of $\Gamma$. We now construct comparison maps $v_\varepsilon$. First, note that $\rho$ is Lipschitz continuous on $N^+$, and define $\rho_1:=\frac{m}{\mu\{{{d}_{\Gamma}(x)}\geq 0\}}\in[m_1, m_2]$. Thus, there exists a point $q\in N^+$ such that $\rho(q)=\rho_1$.

Since $N^+$ is connected, there exists $\xi:[0, 1]\to N^+\in H^1$ such that $\xi(0)=p^+$, $\xi(1)=q$. Let $(p^+, p ^-)\in M^+ \times M^- $ such that $|p^+-p ^-| =d_N$. Define the function $f_{\varepsilon}$ as follows:
 $$f_{\varepsilon}(t)=\begin{cases}p^-& \text { if } t\leq-\varepsilon^{\frac{1}{2}}, \\
\frac{p^{+}+p^{-}}{2}+\alpha_{{\varepsilon}^{-\frac{1}{2}}}\left(\frac{t}{\varepsilon}\right) \frac{p^{+}-p^{-}}{\left|p^{+}-p^{-}\right|} & \text { if } -\varepsilon^{\frac{1}{2}}\leq t\leq\varepsilon^{\frac{1}{2}}, \\
\xi({\varepsilon}^{-\frac{1}{2}}t-1)& \text { if } \varepsilon^{\frac{1}{2}}\leq t \leq 2\varepsilon^{\frac{1}{2}}, \\
q & \text { if } 2\varepsilon^{\frac{1}{2}}\leq t,
\end{cases} $$
where $\alpha_{{\varepsilon}^{-\frac{1}{2}}}$ is defined in Lemma \ref{lemma alpha properties}. Finally, define
$$
v_\varepsilon(x)=f_{\varepsilon}\left({{{d}_{\Gamma}(x)}-\tau_\varepsilon\varepsilon^{\frac{1}{2}}}\right), 
$$
where $\tau_\varepsilon$ will be chosen to satisfy the mass constraint  $\int_{\Omega}\rho\left(v_{\varepsilon}\right) dx=m $. By construction, it follows that $ v_\varepsilon\in H^1(\Omega)$.

$Step$ 2. To determine $\tau_\varepsilon$, we first observe the following. For $ \rho\left(v_\varepsilon\right)$, we have
\begin{equation}
 \rho\left(v_\varepsilon\right)= \begin{cases}0 & \text { on } \left\{ {{d}_{\Gamma}(x)}-\tau_\varepsilon\varepsilon^{\frac{1}{2}}\leq-\varepsilon^{\frac{1}{2}}\right\}, \\ \rho_1 & \text { on }\left\{ {{d}_{\Gamma}(x)}-\tau_\varepsilon\varepsilon^{\frac{1}{2}}\geq2\varepsilon^{\frac{1}{2}}\right\}.\end{cases}\label{v vaule}
\end{equation}
 Since $\partial O$ is $C^2$, by Lemma \ref{approxite}, we have
\begin{equation}
\begin{aligned}
&\mu\left\{ {{d}_{\Gamma}(x)}\geq(\tau_\varepsilon+2)\varepsilon^{\frac{1}{2}}\right\}\\
=&\mu\{{{d}_{\Gamma}(x)}\geq 0\}+H^{n-1}( \{{{d}_{\Gamma}(x)}=0\})(-\tau_\varepsilon-2)\varepsilon^{\frac{1}{2}}+O(\tau_\varepsilon^2\varepsilon)+O(\varepsilon), 
\end{aligned}\label{right area volum}
\end{equation}
and
\begin{equation}
  \begin{aligned}
     &\mu\left\{ (\tau_\varepsilon+2)\varepsilon^{\frac{1}{2}}\geq {{d}_{\Gamma}(x)}\geq(\tau_\varepsilon-1)\varepsilon^{\frac{1}{2}}\right\}\\
 =&3H^{n-1}( \{{{d}_{\Gamma}(x)}=0\})\varepsilon^{\frac{1}{2}}+O(\tau_\varepsilon^2\varepsilon)+O(\varepsilon).
  \end{aligned}\label{middle area volum}
\end{equation}
From the definition of $v_\varepsilon$ and $\rho\in Lip(\mathbb{R}^k)$, there exists $C_0>0$ such that
\begin{equation}
\left\|\rho(v_\varepsilon(\cdot))\right\|_{\infty} \leq C_{0}. \label{v bound}
\end{equation}
 Thus, combining (\ref{v vaule}), (\ref{right area volum}), (\ref{middle area volum}) and (\ref{v bound}) with $\rho_1\mu\{d_\Gamma(x)\geq 0\}=m$, we obtain
\begin{equation} \begin{aligned}
  &\int_{\Omega}\rho\left(v_\varepsilon\right)dx-m \\
  \geq& \rho_1\mu\left\{ {{d}_{\Gamma}(x)}\geq(\tau_\varepsilon+2)\varepsilon^{\frac{1}{2}}\right\}\\
  &+\int_{\left\{ (\tau_\varepsilon+2)\varepsilon^{\frac{1}{2}}\geq {{d}_{\Gamma} (x)}\geq(\tau_\varepsilon-1)\varepsilon^{\frac{1}{2}}\right\}} \rho\left(v_\varepsilon\right)dx-\rho_1\mu\{d_\Gamma(x)\geq 0\} \\ 
  \geq&H^{n-1}( \{{{d}_{\Gamma}(x)}=0\})((-\tau_\varepsilon-2)\rho_1-3C_0)\varepsilon^{\frac{1}{2}}+O(\tau_\varepsilon^2\varepsilon)+O(\varepsilon).\label{3.15}
\end{aligned}
\end{equation}
Similar to the above argument,
\begin{equation}
\int_{\Omega}\rho\left(v_\varepsilon\right)dx-m\leq H^{n-1}( \{{{d}_{\Gamma}(x)}=0\})\left((-\tau_\varepsilon-2)\rho_1+3C_0\right)\varepsilon^{\frac{1}{2}}+O(\tau_\varepsilon^2\varepsilon)+O(\varepsilon).\label{3.16}
\end{equation}
From the above argument, there exists $\tau_{1}>0$ (independent of $\varepsilon$) by further decreasing $\varepsilon$ such that the right-hand side in (\ref{3.15}) is positive for $\tau_\varepsilon \leq-\tau_{1}$, while the right-hand side of (\ref{3.16}) is negative for $\tau_\varepsilon \geq \tau_{1}$. Note that $\int_{\Omega}\rho\left(v_{\varepsilon}\right)dx$ is continuous for $\tau_\varepsilon$. Therefore, there exists $\tau_\varepsilon \in\left[-\tau_{1}, \tau_{1}\right]$ such that $\int_{\Omega}\rho\left(v_{\varepsilon}\right) dx=m $.

$Step$ 3. We now estimate the energy of $v_\varepsilon$ on $\Omega$ as follows:

First, we estimate the energy in the regions$$\left\{{{d}_{\Gamma}(x)}\leq-\varepsilon^{\frac{1}{2}}+\tau_\varepsilon \varepsilon^{\frac{1}{2}}\right\}\cup\left\{ {{d}_{\Gamma}(x)}\geq 2\varepsilon^{\frac{1}{2}}+\tau_\varepsilon \varepsilon^{\frac{1}{2}}\right\}.$$
From the construction of $v_{\varepsilon}$, we have
\begin{equation}
 \int_{\left\{{{d}_{\Gamma}(x)}\leq-\varepsilon^{\frac{1}{2}}+\tau_\varepsilon \varepsilon^{\frac{1}{2}}\right\}\cup\left\{ {{d}_{\Gamma}(x)}\geq 2\varepsilon^{\frac{1}{2}}+\tau_\varepsilon \varepsilon^{\frac{1}{2}}\right\}}\left(|\nabla v_\varepsilon(x)|^{2}+\frac{1}{\varepsilon^{2}} F(v_\varepsilon(x))\right)dx=0, \label{3.17}
\end{equation}
and
\begin{equation}
 F(v_\varepsilon(x))=0\quad\text{ on }\left\{\varepsilon^{\frac{1}{2}}+\tau_\varepsilon \varepsilon^{\frac{1}{2}}\leq {{d}_{\Gamma}(x)}\leq 2\varepsilon^{\frac{1}{2}}+\tau_\varepsilon \varepsilon^{\frac{1}{2}}\right\}.\label{fact}
\end{equation}
 Federer's co-area formula and the formula of change of variables, combined with $|\nabla d_\Gamma(x)|=1$, imply

\begin{equation}
\begin{aligned}
 & \int_{\left\{\varepsilon^{\frac{1}{2}}+\tau_\varepsilon \varepsilon^{\frac{1}{2}}\leq {{d}_{\Gamma}(x)}\leq 2\varepsilon^{\frac{1}{2}}+\tau_\varepsilon \varepsilon^{\frac{1}{2}}\right\}}\left(| \nabla v_\varepsilon(x)|^{2}+\frac{1}{\varepsilon^{2}} F(v_\varepsilon(x))\right) d x\\
 =& \int_{\left\{\varepsilon^{\frac{1}{2}}+\tau_\varepsilon \varepsilon^{\frac{1}{2}}\leq {{d}_{\Gamma}(x)}\leq 2\varepsilon^{\frac{1}{2}}+\tau_\varepsilon \varepsilon^{\frac{1}{2}}\right\}}\left(|\nabla v_\varepsilon(x)|^{2}+\frac{1}{\varepsilon^{2}} F(v_\varepsilon(x))\right)|\nabla {{d}_{\Gamma}(x)}| d x\\
 =&\int_{\varepsilon^{\frac{1}{2}}+\tau_\varepsilon \varepsilon^{\frac{1}{2}}}^{2\varepsilon^{\frac{1}{2}}+\tau_\varepsilon \varepsilon^{\frac{1}{2}}}\left(\left|f_{\varepsilon}^\prime\left({s-\tau_\varepsilon\varepsilon^{\frac{1}{2}}}\right)\right|^2+\frac{1}{\varepsilon^{2}} F(f_\varepsilon({s-\tau_\varepsilon\varepsilon^{\frac{1}{2}}}))\right)\\&\times H^{n-1}(\left\{{{d}_{\Gamma}(x)}=s\right\})ds\\
 \leq&C \int_{\varepsilon^{\frac{1}{2}}+\tau_\varepsilon \varepsilon^{\frac{1}{2}}}^{2\varepsilon^{\frac{1}{2}}+\tau_\varepsilon \varepsilon^{\frac{1}{2}}}\left|\frac{\xi^{\prime}\left({{\varepsilon}^{-\frac{1}{2}}s-\tau_\varepsilon-1}\right)}{\varepsilon^{\frac{1}{2}}}\right|^2ds\\
 \leq& C\varepsilon^{\frac{1}{2}}\left(\frac{\max_{0\leq s\leq 1 }|\xi^{\prime}(s)|}{\varepsilon^{\frac{1}{2}}}\right)^2 \\
 =&O({\varepsilon}^{-\frac{1}{2}}), 
\end{aligned}\label{3.18}
\end{equation}
where we have used $H^{n-1}(\left\{d_\Gamma(x)=s\right\})\leq C$ (when $s$ is small enough and (\ref{fact})). 

 For $\varepsilon$ sufficiently small, applying Lemma \ref{technical lemma 2} gives
 \begin{equation}
 \max _{-\varepsilon^{\frac{1}{2}}+\tau_\varepsilon \varepsilon^{\frac{1}{2}}\leq s\leq \varepsilon^{\frac{1}{2}}+\tau_\varepsilon \varepsilon^{\frac{1}{2}}} H^{n-1}\left(\{{{d}_{\Gamma}(x)}=s\}\right)\leq I_{\Omega}({\sigma_m})+2\delta.\label{error}
 \end{equation}

Using the bound above, Federer's co-area formula, and the formula of change of variables, we can compute

\begin{equation}
\begin{aligned}
 & \int_{\left\{-\varepsilon^{\frac{1}{2}}+\tau_\varepsilon \varepsilon^{\frac{1}{2}}\leq {{d}_{\Gamma}(x)}\leq \varepsilon^{\frac{1}{2}}+\tau_\varepsilon \varepsilon^{\frac{1}{2}}\right\}}\left(|\nabla v_\varepsilon(x)|^{2}+\frac{1}{\varepsilon^{2}} F(v_\varepsilon(x))\right) d x\\
 \leq &\left(\max _{-\varepsilon^{\frac{1}{2}}+\tau_\varepsilon \varepsilon^{\frac{1}{2}}\leq s\leq \varepsilon^{\frac{1}{2}}+\tau_\varepsilon \varepsilon^{\frac{1}{2}}} H^{n-1}\left(\{{{d}_{\Gamma}(x)}=s\}\right)\right) \\&\times
 \frac{1}{\varepsilon}\int_{{-\varepsilon}^{-\frac{1}{2}}}^{ {\varepsilon}^{-\frac{1}{2}}} {\left[\left|{\alpha}_{\varepsilon^{-\frac{1}{2}}}^{\prime}(s)\right|^{2}+\overline{F}\left({\alpha}_{\varepsilon^{-\frac{1}{2}}}(s)\right)\right]}ds\\
 \leq&\frac{I_{\Omega}({\sigma_m})+2\delta}{\varepsilon}c_0^F;
\end{aligned}\label{3.19}
\end{equation}
in the last step we have used the properties of $\alpha_{2L}$ from Lemma \ref{lemma alpha properties}.

Finally, we rescale the functional by a factor $\frac{1}{\varepsilon}$. Combining (\ref{3.17}), (\ref{3.18}) and (\ref{3.19}), we have
\begin{equation}
\limsup_{\varepsilon \to 0} \varepsilon \mathbf{E}_\varepsilon\leq\limsup_{\varepsilon\to0}\int_{\Omega}\left(\varepsilon|\nabla v_\varepsilon(x)|^{2}+\frac{1}{\varepsilon} F(v_\varepsilon(x))\right) d x\leq (I_{\Omega}({\sigma_m})+2\delta)c_0^F.
\end{equation}
Since $\delta>0$ is arbitrary, the proof is complete.
\end{proof}

\subsubsection{Refined upper bound estimates of \texorpdfstring{$\mathbf{E}_\varepsilon $}{} }
\quad In this subsubsection, we aim to derive a refined upper estimate for the upper bound of $\mathbf{E}_\varepsilon $ by constructing comparison maps. Unlike the proof in Proposition \ref{upper}, here we need a fine upper bound. Therefore, the approximation technique of Lemma \ref{technical lemma 1} fails. Under condition (\ref{G1}), the zero constant mean curvature of $\Gamma$ plays a crucial role in obtaining the refined estimate of $\mathbf{E}_\varepsilon $.

\begin{proposition}
 Under the same assumptions of Theorem \ref{theorem 2}, there exists a constant $C>0 $ and $\varepsilon_0>0$ such that for $0<\varepsilon\leq \varepsilon_0$, the following holds 
 \begin{equation}
\mathbf{E}_{\varepsilon} \leq \frac{c_0^FI_{\Omega}({\sigma_m})}{\varepsilon}+C. \label{3.03}
\end{equation}
\label{fine upper}
\end{proposition}
\begin{proof}
\text{{}}

$Step$ 1. The first step is to construct the function $v_{\varepsilon}$ satisfying the constraint $\int_{\Omega}\rho\left(v_{\varepsilon}\right) dx=m $. For $\Gamma$ (as defined in (\ref{gammer})), the following regularity properties hold (\cite{sternberg1999connectivity}, Theorem 2.1). There exists a closed subset of $\Gamma$ denoted by $\operatorname{sing} \Gamma$, whose Hausdorff dimension is at most $n-8$ for $n>7$ and which is an empty set for $n\leq 7$, such that on the compliment of this set, $\operatorname{reg} \Gamma:=\Gamma \backslash \operatorname{sing} \Gamma$. We have:

(i) $\Gamma$ consists of a finite union of $(N -1)$-dimensional regular surfaces \\ $\Gamma^{(1)}, \cdots, \Gamma^{(k)}$ and has zero constant mean curvature. 

(ii) For every $x \in\operatorname{reg} \Gamma\cap
\partial\Omega$, $\Gamma$ can be represented locally as the graph of a $C^{2, \alpha}$ function which meets $\partial\Omega$ orthogonally, where $0<\alpha<1$. 

To define a $C^{2, \alpha}$ signed distance function, we extend $\Gamma$ outside of $\Omega$ locally via reflection (we also call it $\Gamma$). 
Define the signal distance function ${d}_\Gamma(x)$ to $\Gamma$ (see Lemma \ref{technical lemma 2}). Let $S(t)= {H}^{n-1}(\{{d}_\Gamma(x)=t\})$. From (\cite{leoni2016second}, Lemma 5.4), $S$ is twice differentiable at $t=0$ and satisfies
\begin{equation}
 S^{\prime}(0)=(n-1)\kappa_\Gamma S(0)=0, 
\end{equation}
where $\kappa_\Gamma$ is the mean curvature of $\Gamma$. Thus, we have the Taylor expansion, as $t\to0$, 
\begin{equation}
 S(t)=S(0)+O(t^2).\label{expansion}
\end{equation}
Let $(p^+, p ^-)\in M^+ \times M^- $ such that $|p^+-p ^-| =d_N$. From $\mu(\Omega^+)=\mu\{ {d}_\Gamma(x)\geq0\}\in(\frac{m}{m_2},\frac{m}{m_1})$ (see (\ref{gammer})) and $\rho(N^+)=[m_1,m_2]$, there exists $u_0\in H^1:\Omega^+\to N^+$ such that
\begin{equation}
\begin{aligned}
 &1.\text{ }\int_{\{ {d}_\Gamma(x)\geq0\}}\rho(u_0(x))dx=m, \\
 &2.\text{ }u_0(x)=p^+\text{ on } \{0\leq {d}_{\Gamma}(x)\leq \delta\}\text{ for some small } \delta>0.\\
\end{aligned}\label{u_0}
\end{equation}
We extend $u_0$ to $\Omega$ by defining 
\begin{equation}
 \widetilde{u}_0=
 \begin{cases}u_0 & \text{ if } x\in\Omega^+, \\
 p^+& \text{ if } x\in \Omega\backslash\Omega^+. 
 \end{cases}
\end{equation}

For $\gamma\in(\frac{1}{2}, 1)$, define the following subsets of $\Omega$,
\begin{equation}
 \{{d}_{\Gamma}(x)\leq(-1+\tau_\varepsilon)\varepsilon^\gamma\} =\Omega_1^\varepsilon, 
\end{equation}
\begin{equation}
 \{(-1+\tau_\varepsilon)\varepsilon^\gamma\leq{d}_{\Gamma}(x)\leq(1+\tau_\varepsilon)\varepsilon^\gamma\}=\Omega_2^\varepsilon, 
\end{equation}
\begin{equation}
 \{{d}_{\Gamma}(x)\geq(1+\tau_\varepsilon)\varepsilon^\gamma\}=\Omega_3^\varepsilon.
\end{equation}

We now construct comparison maps $v_\varepsilon:\Omega\to\mathbb{R}^k$. Using Lemma \ref{lemma alpha properties}, with $L=\varepsilon^{\gamma-1}$ and the function $\widetilde{u}_0$, we define $ v_\varepsilon$ as:
\begin{equation}
 v_\varepsilon(x)=
 \begin{cases}
 p^-&\text{ on } \Omega_1^\varepsilon, \\
 \frac{p^{+}+p^{-}}{2}+\alpha_{\varepsilon^{\gamma-1}}\left(\frac{{d}_{\Gamma}(x)-\tau_\varepsilon\varepsilon^\gamma}{\varepsilon}\right) \frac{p^{+}-p^{-}}{\left|p^{+}-p^{-}\right|} & \text { on } \Omega_2^\varepsilon, \\
 \widetilde{u}_0 &\text{ on } \Omega_3^\varepsilon, 
 \end{cases}
\end{equation}
where $\tau_\varepsilon$ will be determined later. 

By construction, $v_\varepsilon\in H^1(\Omega, \mathbb{R}^k)$ if the following condition is satisfied
\begin{equation}
 (1+\tau_\varepsilon)\varepsilon^\gamma\leq\delta.\label{connect1}
\end{equation}

$Step$ 2.
 Now, we determine $\tau_\varepsilon$ (uniformly bounded for $\varepsilon$) such that the mass constraint is satisfied.

 Let $ \tau_\varepsilon\leq\delta\varepsilon^{-\gamma}-1\label{connect}$ to satisfied the condition (\ref{connect1}). First, by the construction of $v_\varepsilon$, we know

\begin{equation}
 \rho\left(v_{\varepsilon}\right)= \begin{cases}
0& \text { on } \Omega_1^\varepsilon, \\ 
 \rho(\widetilde{u}_0) & \text { on }\Omega_3^\varepsilon.\end{cases}\label{3.2.311}
\end{equation}

Using the expansion (\ref{expansion}), we have 
\begin{equation}
 \mu(\Omega_2^\varepsilon)= 2H^{n-1}(\Gamma)\varepsilon^{\gamma}+O(((1+|\tau_\varepsilon|)\varepsilon^\gamma)^2), 
\end{equation}
This, combined $\left\|\rho(v_\varepsilon(\cdot))\right\|_{\infty} \leq C_{0}$, we obtain
\begin{equation}
\begin{aligned}
\int_{\Omega_2^\varepsilon}\rho\left(v_{\varepsilon}\right)dx
\leq&C_{0}\left(H^{n-1}(\Gamma)\varepsilon^{\gamma}+O(((1+|\tau_\varepsilon|)\varepsilon^\gamma)^2)\right)
\end{aligned}
\end{equation}
Now, let us split this into two cases by $\tau_\varepsilon$:

For $\tau_\varepsilon\geq-1$, use the expansion (\ref{expansion}), we have

\begin{equation}
\begin{aligned}
 & \mu\left(\{0\leq{d}_{\Gamma}(x)\leq(1+\tau_\varepsilon)\varepsilon^\gamma\}\right)\\
 =&(1+\tau_\varepsilon)\varepsilon^\gamma H^{n-1}(\Gamma)+O(((1+\tau_\varepsilon)\varepsilon^\gamma)^2).
\end{aligned}
\end{equation}
Since $\rho(\widetilde{u}_0)\in[m_1, m_2] \text{ on } \Omega$, by the construction of (\ref{u_0}), we compute
\begin{equation}
 \begin{aligned}
 &\int_\Omega \rho(v_\varepsilon)dx-m\\
 = & \int_\Omega \rho(v_\varepsilon)dx-\int_{\{ {d}_\Gamma(x)\geq0\}}\rho(\widetilde{u}_0(x))dx\\
 \leq &C_0H^{n-1}(\Gamma)\varepsilon^{\gamma}+O(((1+|\tau_\varepsilon|)\varepsilon^\gamma)^2)-\int_{\{0\leq{d}_{\Gamma}(x)\leq(1+\tau_\varepsilon)\varepsilon^\gamma\}}m_1dx\\
 \leq &\left(C_0-(1+\tau_\varepsilon)m_1\right)H^{n-1}(\Gamma)\varepsilon^{\gamma}+O(((1+|\tau_\varepsilon|)\varepsilon^\gamma)^2).
 \end{aligned}\label{mass sup}
\end{equation}

For $\tau_\varepsilon\leq-1$, similar to the above argument, we obtain
\begin{equation}
 \begin{aligned}
 &\int_\Omega \rho(v_\varepsilon)dx-m\\
 \geq &\int_{\{(-1+\tau_\varepsilon)\varepsilon^\gamma\leq{d}_{\Gamma}(x)\leq0\}}\rho(p^+)dx-C_0\varepsilon^{\gamma}+O(((1+|\tau_\varepsilon|)\varepsilon^\gamma)^2)\\
 \geq &\left((1-\tau_\varepsilon)m_1- C_0\right)H^{n-1}(\Gamma)\varepsilon^{\gamma}+O(((1+|\tau_\varepsilon|)\varepsilon^\gamma)^2).
 \end{aligned}\label{mass inf}
\end{equation}

So, there exists $\tau_{1}$,  which is independent of $\varepsilon$, such that $\delta\varepsilon^{-\gamma}-1>\tau_{1}>0$ by further decreasing $\varepsilon$ and the r.h.s. in (\ref{mass sup}) is negative for $\tau_\varepsilon\leq-\tau_{1}$ while the r.h.s. of (\ref{mass inf}) is positive for $\tau_\varepsilon\geq \tau_{1}$. Note that $\int_{\Omega}\rho\left(v_{\varepsilon}\right)dx$ is continuous for $\tau_\varepsilon$. Therefore, there exists $\tau_\varepsilon \in\left[-\tau_{1}, \tau_{1}\right]$ for which $\int_{\Omega}\rho\left(v_{\varepsilon}\right) dx=m $.

$Step$ 3. We now estimate the energy of $v_\varepsilon$ on $\Omega$ as follows:
\begin{equation}
 \begin{aligned}
 &\int_{ \Omega}\left(|\nabla v_\varepsilon(x)|^{2}+\frac{1}{\varepsilon^{2}} F(v_\varepsilon(x))\right) d x\\
 =&\left\{\int_{\Omega_1^\varepsilon}+\int_{\Omega_2^\varepsilon}+\int_{\Omega_3^\varepsilon}\right\}\left(|\nabla v_\varepsilon(x)|^{2}+\frac{1}{\varepsilon^{2}} F(v_\varepsilon(x))\right) d x\\
 =& \mathrm{I}_\varepsilon+ \mathrm{II}_\varepsilon+ \mathrm{III}_\varepsilon.
 \end{aligned}
\end{equation}
First, for the term $\mathrm{I}_\varepsilon$, since $v_\varepsilon=p^-$ on $\Omega_1^\varepsilon$, which implies $\nabla v_\varepsilon
=0$ and $F(v_\varepsilon)=0 $. Thus
\begin{equation}
 \mathrm{I}_\varepsilon=0.\label{I}
\end{equation}
For the term $\mathrm{III}_\varepsilon$, since $v_\varepsilon=\widetilde{u}_0\in H^1(\Omega,N^+)$ on $\Omega_3^\varepsilon$ , we have
\begin{equation}
\begin{aligned}
   \mathrm{III}_\varepsilon
   &=\int_{\Omega_3^\varepsilon}|\nabla v_\varepsilon(x)|^{2}dx
   =\int_{\Omega_3^\varepsilon}|\nabla\widetilde{u}_0|^2dx\\
   &\leq \int_{\Omega^+}|\nabla u_0(x)|^{2}dx\leq C.\label{III}
\end{aligned}
\end{equation}
Now we estimate $\mathrm{II}_\varepsilon$. Since
\begin{equation}
 |\nabla {d}_\Gamma(x)|=1\quad\text{ for } x\in\Omega_2^\varepsilon, 
\end{equation}
we have
\begin{equation}
\begin{aligned}
 \mathrm{II}_\varepsilon=& \int_{ \Omega_2^\varepsilon}\left(|\nabla v_\varepsilon(x)|^{2}+\frac{1}{\varepsilon^{2}} F(v_\varepsilon(x))\right) d x\\=&\frac{1}{\varepsilon^2}\int_{\Omega_2^\varepsilon} \Bigg[\left|{\alpha}_{\varepsilon^{\gamma-1}}^{\prime}\left(\frac{{d}_{\Gamma}(x)-\tau_\varepsilon\varepsilon^\gamma}{\varepsilon}\right)\right|^{2}|\nabla {d}_{\Gamma}(x)|^2\\ &+ \overline{F}\left({\alpha}_{\varepsilon^{\gamma-1}}\left(\frac{{d}_{\Gamma}(x)-\tau_\varepsilon\varepsilon^\gamma}{\varepsilon}\right)\right)\Bigg]ds\\
 =&\frac{1}{\varepsilon^2}\int_{\Omega_2^\varepsilon} \Bigg[\left|{\alpha}_{\varepsilon^{\gamma-1}}^{\prime}\left(\frac{{d}_{\Gamma}(x)-\tau_\varepsilon\varepsilon^\gamma}{\varepsilon}\right)\right|^{2}\\&+\overline{F}\left({\alpha}_{\varepsilon^{\gamma-1}}\left(\frac{{d}_{\Gamma}(x)-\tau_\varepsilon\varepsilon^\gamma}{\varepsilon}\right)\right)\Bigg] |\nabla {d}_{\Gamma}(x)|ds.
\end{aligned}\label{II}
\end{equation}
Since (\ref{expansion}) gives
\begin{equation}
 {H}^{n-1}(\{{d}_\Gamma(x)=s\}\leq{H}^{n-1}(\Gamma)+Cs^2, 
\end{equation}
we have, by both the Federer's co-area formula and the formula of change of variables, 
\begin{equation}
\begin{aligned}
 \mathrm{II}_\varepsilon=&\frac{1}{\varepsilon^2}\int_{(-1+\tau_\varepsilon)\varepsilon^\gamma}^{(1+\tau_\varepsilon)\varepsilon^\gamma} \Bigg[\left|{\alpha}_{\varepsilon^{\gamma-1}}^{\prime}\left(\frac{s-\tau_\varepsilon\varepsilon^\gamma}{\varepsilon}\right)\right|^{2}\\&+\overline{F}\left({\alpha}_{\varepsilon^{\gamma-1}}\left(\frac{s-\tau_\varepsilon\varepsilon^\gamma}{\varepsilon}\right)\right)
 \Bigg]\times
 {H}^{n-1}(\{{d}_\Gamma(x)=s\})ds\\
 \leq& \frac{1}{\varepsilon^2}\int_{-\varepsilon^\gamma}^{\varepsilon^\gamma} {\left[\left|{\alpha}_{\varepsilon^{\gamma-1}}^{\prime}\left(\frac{s}{\varepsilon}\right)\right|^{2}+\overline{F}\left({\alpha}_{\varepsilon^{\gamma-1}}\left(\frac{s}{\varepsilon}\right)\right)\right]} \\&\times
 \left({H}^{n-1}(\Gamma)+C(|s|+|\tau_\varepsilon|\varepsilon^{\gamma})^2\right)ds\\
 =&\frac{H^{n-1}(\Gamma)}{\varepsilon} \int_{-\varepsilon^{\gamma-1}}^{ \varepsilon^{\gamma-1}} {\left[\left|{\alpha}_{\varepsilon^{\gamma-1}}^{\prime}(s)\right|^{2}+\overline{F}\left({\alpha}_{\varepsilon^{\gamma-1}}(s)\right)\right]}ds\\
 &+
 C(| \tau_\varepsilon|+1)^2\varepsilon^{2\gamma-1}\int_{-\varepsilon^{\gamma-1}}^{ \varepsilon^{\gamma-1}} {\left[\left|{\alpha}_{\varepsilon^{\gamma-1}}^{\prime}(s)\right|^{2}+\overline{F}\left({\alpha}_{\varepsilon^{\gamma-1}}(s)\right)\right]}ds\\ =&\mathrm{IV}_\varepsilon+\mathrm{V}_\varepsilon.\\
\end{aligned}
\end{equation}
For the term $\mathrm{IV}_\varepsilon$, by Lemma \ref{lemma alpha properties} (ii), 
\begin{equation}
 \mathrm{IV}_\varepsilon\leq \frac{H^{n-1}(\Gamma)}{\varepsilon}(c_0^F+c_2e^{-c_1\varepsilon^{\gamma-1}})\leq \frac{H^{n-1}(\Gamma)c_0^F}{\varepsilon}+C\varepsilon.\label{IV}
\end{equation}
For the term $\mathrm{V}_\varepsilon$, we also use Lemma \ref{lemma alpha properties} (ii) and $\frac{1}{2}<\gamma<1$ to get
\begin{equation}
 \mathrm{V}_\varepsilon\leq C\varepsilon^{2\gamma-1}\leq o(1).\label{V}
\end{equation}
Finally, combining (\ref{I}), (\ref{II}), (\ref{III}), with (\ref{IV}), we have
\begin{equation}
 \int_{ \Omega}\left(|\nabla v_\varepsilon(x)|^{2}+\frac{1}{\varepsilon^{2}} F(v_\varepsilon(x))\right) d x\leq \frac{H^{n-1}(\Gamma)c_0^F}{\varepsilon}+C.
\end{equation}
The proof of Proposition \ref{fine upper} is complete. 
\end{proof}

\subsection{Lower bound estimates of \texorpdfstring{$\mathbf{E}_\varepsilon $}{} \label{section 3.2}}
This subsection is devoted to the proof of lower bound estimates for $\mathbf{E}_\varepsilon $. 
The following property (see, \cite{maggi2012sets}) is important for our analysis:
\begin{equation*}
I_{\Omega}(t) \text { is locally Lipschitz on }(0, 1).
\end{equation*}

For any fixed $\delta \in\left[0, \frac{d_N}{2} \right]$, define the two disjoint measurable subsets
$$
\Omega_{\varepsilon, \delta}^{ \pm}=\left\{x \in \Omega: d\left(u_{\varepsilon}(x), N^{ }\right)\leq\delta\right\}
$$
and the complement subset
$$
E_{\varepsilon, \delta}=\Omega \backslash\left(\Omega_{\varepsilon, \delta}^{+} \cup \Omega_{\varepsilon, \delta}^{-}\right) \equiv\left\{x \in \Omega: {d}\left(u_{\varepsilon}(x), N\right) > \delta\right\}.
$$
Since $\{u_\varepsilon\}\in H^1(\Omega)$, it follows that $d(u_\varepsilon, N)\in H^1(\Omega)$, which implies that $\Omega_{\varepsilon, \delta}^{ \pm}$ is a set $E$ of finite perimeter for almost all $t\in\mathbb{R}^+$.  

For simplicity of notation, we denote $d(u_\varepsilon(x)):=d\left(u_{\varepsilon}(x),N\right)$.We denote the reduced boundary of a set $E$ of finite perimeter by $\partial^*E$. Recalling from \cite{maggi2012sets}, we have 
\begin{equation}
  H^{n-1}(\partial^*E\cap\Omega)=\operatorname{Per}_\Omega E=D_{\chi_E}(\Omega).
\end{equation}
 where ${\chi_E}$ is the characteristic function of $E$. 
 We also use Federer's co-area formula for a mapping $f$ belonging to the Sobolev class $H^1(\Omega, \mathbb{R})$ \cite{Mal2001TheCF}.
For $f \in L^1_{\text{loc}}(\Omega)$, a function $\tilde{f}$ is said to be a \textit{precise representative} of $f$ if
\begin{equation}
\tilde{f}(x) := \lim_{r \to 0} \frac{1}{|B(x, r)|} \int_{B(x, r)} f(y) \, dy
\end{equation}
at all points $x$ where this limit exists. From the Lebesgue differentiation theorem, we may always assume $d(u_\varepsilon, N)$ is precisely represented by modifying on a set of Lebesgue measure zero, which does not affect our conclusions. 
\begin{lemma}
  Suppose that $f\in H^1(\Omega, \mathbb{R})$ is precisely represented and let $g:\mathbb{R}\to\mathbb{R} $ be smooth. Then $f^{-1}(t) $ is countably $H^{n-1}$-rectifiable for almost all $t\in\mathbb{R}$. Moreover, the following co-area formula holds for every $L^n$ measurable set E:
 \begin{equation}
   \int_E g(f(x))|\nabla f(x)|dx=\int_{\mathbb{R}}g(t)H^{n-1}(\{f^{-1}(t)\}\cap E)dt, \label{Federer's co-area for sobolev}
 \end{equation}
 and
 \begin{equation}
   H^{n-1}(\{f^{-1}(t)\}\cap\Omega)=H^{n-1}(\partial^*\{f\leq t)\}\cap\Omega\})\label{equivalence of topological and reduced boundaries}
 \end{equation}
 for almost all $t\in\mathbb{R}$.
\end{lemma}
\begin{proof}
  For (\ref{Federer's co-area for sobolev}), see (\cite{Mal2001TheCF}, Theorem 1.1). We prove (\ref{equivalence of topological and reduced boundaries}). By (\cite{Mal2001TheCF}, Lemma 3.2), there exist Lipschitz functions $f_k: \mathbb{R}^n \to \mathbb{R}$ and disjoint Borel subsets $E_k$ of $\mathbb{R}^n$ such that $f = f_k$ on $E_k$ and $L^n \left(\Omega\setminus \bigcup E_k \right) = 0$.
  That $f^{-1}(t)$ is countably $H^{n-1}$ rectifiable for almost all $y \in \mathbb{R}^m$ follows from the fact that for each $k$, $E_k \cap f^{-1}(t) = E_k \cap f_k^{-1}(t)$ is countably $H^{n-1}$-rectifiable for almost all $t \in \mathbb{R}$. Using $H^{n-1}(E \cap f^{-1}(t)) = 0$ for almost all $t$ whenever $L^n(E) = 0$ (see (\cite{Mal2001TheCF}, Theorem 3.3) and $H^{n-1}$-equivalence of topological and reduced boundaries of a.e. lower-level set of a Lipschitz function (\cite{maggi2012sets}, Theorem 18.1), we obtain for almost all $t\in\mathbb{R}$,
\begin{equation}
  \begin{aligned}
      H^{n-1}(\{f^{-1}(t)\}\cap\Omega)&=\sum_k H^{n-1}(E_k\cap\Omega \cap f_k^{-1}(t))\\
      &=\sum_k D_{\chi_{ \{f_k\leq t)\}}}(E_k\cap\Omega)\\
      &=\sum_k D_{\chi_{ \{f\leq t)\}}}(E_k\cap\Omega)\\
      &=D_{\chi_{\{f\leq t)\}}}(\Omega)\\
      &=H^{n-1}(\partial^*\{f\leq t)\}\cap\Omega\}).
  \end{aligned}
  \end{equation}
\end{proof}
We need a crude estimate similar to that in (\cite{lin2012phase}, Lemma 4.1):
\begin{proposition}
   There exists $C=C(N)>0$ such that for any $0<\delta \leq \delta_{N}$, it holds that
\begin{equation}
\begin{aligned}
& \int_{\Omega}\left(\left|\nabla u_{\varepsilon}\right|^{2}+\frac{1}{\varepsilon^{2}} F\left(u_{\varepsilon}\right)\right) d x \\
\geq&(1-C \delta) \int_{\Omega_{\varepsilon, \delta}^{+} \cup \Omega_{\varepsilon, \delta}^{-}}\left|\nabla\left(\Pi\left(u_{\varepsilon}\right)\right)\right|^{2} d x \\
 &+ \frac{2}{\varepsilon} \int_{0}^{\frac{d_{N}}{2}} \left( H^{n-1}(\partial^* \Omega_{\varepsilon, \lambda}^+ \cap \Omega) + H^{n-1}(\partial^* \Omega_{\varepsilon, \lambda}^- \cap \Omega) \right) \sqrt{f\left(\lambda^{2}\right)} d \lambda .\label{crude estimate a}
\end{aligned}
\end{equation}
\label{crude estimate}
\end{proposition}
\begin{proof} 
For $0<\delta \leq \delta_{N}$, recall that, for $x \in \Omega_{\varepsilon, \delta}^{+} \cup \Omega_{\varepsilon, \delta}^{-}$, 
$$
\begin{aligned}
u_{\varepsilon}(x) & =\Pi\left(u_{\varepsilon}(x)\right)+{d}\left(u_{\varepsilon}(x)\right) v_{\varepsilon}(x), \\
v_{\varepsilon}(x) & =\frac{u_{\varepsilon}(x)-\Pi\left(u_{\varepsilon}(x)\right)}{\left|u_{\varepsilon}(x)-\Pi\left(u_{\varepsilon}(x)\right)\right|}.
\end{aligned}
$$
 For $x \in \Omega_{\varepsilon, \delta}^{+} \cup \Omega_{\varepsilon, \delta}^{-}$, as in (\cite{lin2012phase}, Lemma 4.1), we have
$$
\begin{aligned}
\left|\nabla u_{\varepsilon}(x)\right|^{2}= & \left|\nabla\left(\Pi\left(u_{\varepsilon}(x)\right)\right)\right|^{2}+|\nabla\left({d}\left(u_{\varepsilon}(x)\right)\right|^{2}+{d}^{2}\left(u_{\varepsilon}(x)\right)\left|\nabla\left(v_{\varepsilon}(x)\right)\right|^{2} \\
& -2d\left(u_{\varepsilon}(x)\right) A\left(\Pi\left(u_{\varepsilon}(x)\right)\left(\nabla\left(\Pi\left(u_{\varepsilon}(x)\right)\right), \nabla\left(\Pi\left(u_{\varepsilon}(x)\right)\right)\right)\right. \\
\geq & \left|\nabla\left({d}\left(u_{\varepsilon}(x)\right)\right)\right|^{2}+(1-C \delta)\left|\nabla\left(\Pi\left(u_{\varepsilon}(x)\right)\right)\right|^{2}, 
\end{aligned}
$$
where $A(\cdot)(\cdot, \cdot)$ is the second fundamental form of $N$ and $C=2\|A(\cdot)\|_{L^{\infty}(N)}$. Thus we obtain
\begin{equation}
\begin{aligned}
&\int_{\Omega}  \left(\left|\nabla u_{\varepsilon}\right|^{2}+\frac{1}{\varepsilon^{2}} F\left(u_{\varepsilon}\right)\right) d x \\
\geq&(1-C \delta) \int_{\Omega_{\varepsilon, \delta}^{+} \cup \Omega_{\varepsilon, \delta}^{-}}\left|\nabla\left(\Pi\left(u_{\varepsilon}\right)\right)\right|^{2} d x \\
& +\int_{\Omega_{\varepsilon, \delta}^{+} \cup \Omega_{\varepsilon, \delta}^{-}}\left(\left|\nabla\left({d}\left(u_{\varepsilon}\right)\right)\right|^{2}+\frac{1}{\varepsilon^{2}} F\left(u_{\varepsilon}\right)\right) d x \\
& +\int_{\Omega \backslash\left(\Omega_{\varepsilon, \delta}^{+} \cup \Omega_{\varepsilon, \delta}^{-}\right)}\left(\left|\nabla u_{\varepsilon}\right|^{2}+\frac{1}{\varepsilon^{2}} F\left(u_{\varepsilon}\right)\right) d x .
\end{aligned}\label{crude estimate 1}
\end{equation}
By the Cauchy-Schwarz inequality, co-area formula (\ref{Federer's co-area for sobolev}) and (\ref{equivalence of topological and reduced boundaries}), we estimate
\begin{equation}
  \begin{aligned}
& \int_{\Omega_{\varepsilon, \delta}^{+} \cup \Omega_{\varepsilon, \delta}^{-}}\left(\left|\nabla\left({d}\left(u_{\varepsilon}\right)\right)\right|^{2}+\frac{1}{\varepsilon^{2}} F\left(u_{\varepsilon}\right)\right) d x \\
\geq &\frac{2}{\varepsilon} \int_{\Omega_{\varepsilon, \delta}^{+} \cup \Omega_{\varepsilon, \delta}^{-}} \sqrt{f\left(d^{2}\left(u_{\varepsilon}\right)\right)}\left|\nabla\left({d}\left(u_{\varepsilon}\right)\right)\right| d x \\
 = & \frac{2}{\varepsilon} \int_{0}^{\delta}\left( H^{n-1}(\partial^* \Omega_{\varepsilon, \lambda}^+ \cap \Omega) + H^{n-1}(\partial^* \Omega_{\varepsilon, \lambda}^- \cap \Omega) \right) \sqrt{f\left(\lambda^{2}\right)} d \lambda .
\end{aligned}\label{crude estimate for inner}
\end{equation}
$$
H_{\varepsilon, \delta}^{ \pm}=\left\{x \in \Omega: \delta \leq {d}\left(u_{\varepsilon}(x), N^{ \pm}\right) \leq \frac{d_{N}}{2}\right\} .
$$
Since
$$
{d}\left(u_{\varepsilon}(x)\right)={d}\left(u_{\varepsilon}(x), N^{ \pm}\right) \quad \text { and } \quad\left|\nabla\left({d}\left(u_{\varepsilon}(x)\right)\right)\right| \leq\left|\nabla u_{\varepsilon}(x)\right| \quad \text { for } x \in H_{\varepsilon, \delta}^{ \pm}, 
$$
calculations similar to those for (\ref{crude estimate for inner}), we have
\begin{equation}
  \begin{aligned}
 &\int_{\Omega \backslash\left(\Omega_{\varepsilon, \delta}^{+} \cup \Omega_{\varepsilon, \delta}^{-}\right)}\left(\left|\nabla u_{\varepsilon}\right|^{2}+\frac{1}{\varepsilon^{2}} F\left(u_{\varepsilon}\right)\right) d x \\
\geq &\int_{H_{\varepsilon, \delta}^{+} \cup H_{\varepsilon, \delta}^{-}}\left(\left|\nabla u_{\varepsilon}\right|^{2}+\frac{1}{\varepsilon^{2}} F\left(u_{\varepsilon}\right)\right) d x \\
 \geq& \frac{2}{\varepsilon} \int_{H_{\varepsilon, \delta}^{+} \cup H_{\varepsilon, \delta}^{-}}
\sqrt{f\left(d^{2}\left(u_{\varepsilon}\right)\right)}\left|\nabla u_{\varepsilon}\right| d x \\
 \geq& \frac{2}{\varepsilon} \int_{H_{\varepsilon, \delta}^{+} \cup H_{\varepsilon, \delta}^{-}} \sqrt{f\left(d^{2}\left(u_{\varepsilon}\right)\right)}\left|\nabla\left({d}\left(u_{\varepsilon}\right)\right)\right| d x \\
 =&\frac{2}{\varepsilon} \int_{\delta}^{\frac{d_{N}}{2}} \left( H^{n-1}(\partial^* \Omega_{\varepsilon, \lambda}^+ \cap \Omega) + H^{n-1}(\partial^* \Omega_{\varepsilon, \lambda}^- \cap \Omega) \right)\sqrt{f\left(\lambda^{2}\right)} d \lambda .
\end{aligned}\label{crude estimate for outer}
\end{equation}
Putting (\ref{crude estimate for inner}) and (\ref{crude estimate for outer}) together, we get

\begin{align*}
& \quad \int_{\Omega_{\varepsilon, \delta}^{+} \cup \Omega_{\varepsilon, \delta}^{-}}\left(\left|\nabla\left({d}\left(u_{\varepsilon}\right)\right)\right|^{2}+\frac{1}{\varepsilon^{2}} F\left(u_{\varepsilon}\right)\right) d x \\
& +\int_{\Omega \backslash\left(\Omega_{\varepsilon, \delta}^{+} \cup \Omega_{\varepsilon, \delta}^{-}\right)}\left(\left|\nabla u_{\varepsilon}\right|^{2}+\frac{1}{\varepsilon^{2}} F\left(u_{\varepsilon}\right)\right) d x \\
& \quad \geq \frac{2}{\varepsilon} \int_{0}^{\frac{d_{N}}{2}} \left( H^{n-1}(\partial^* \Omega_{\varepsilon, \lambda}^+ \cap \Omega) + H^{n-1}(\partial^* \Omega_{\varepsilon, \lambda}^- \cap \Omega) \right) \sqrt{f\left(\lambda^{2}\right)} d \lambda .
\end{align*}
This, combined with (\ref{crude estimate 1}) implies (\ref{crude estimate a}). 
\end{proof}
In the following part, we fix $\delta^{*}\in(0,\delta_{N})$ satisfying $C(N) \delta^*<1$, where $C(N)$ is given in Proposition \ref{crude estimate}.  Next, we prove the following proposition:
\begin{proposition}
 For some sufficiently small $\varepsilon_0>0$, there exist $K_\varepsilon$ and $ L_\varepsilon$, which are uniformly bounded for $\varepsilon<\varepsilon_0$, such that
 \begin{equation}
\mu(\Omega_{\varepsilon, \delta^*}^+)+K_\varepsilon\varepsilon^{\frac{1}{2}}\in[\frac{m}{m_2}, \frac{m}{m_1}], \label{l5}
\end{equation}
\begin{equation}
 \mu(\Omega_{\varepsilon, \delta^*}^-)+L_\varepsilon\varepsilon^{\frac{1}{2}}\in[1-\frac{m}{m_1}, 1-\frac{m}{m_2}].
\end{equation}

\label{proposition 3.1}
\end{proposition}
\begin{proof}
Let $\delta_{1}(\delta^*)>0$ be sufficiently small, by properties (\ref{H1}), we have
\begin{equation}
E_{\varepsilon, \delta^*} \subset\left\{F(u_\varepsilon) \geq \delta_{1}\right\}, \label{l6}
\end{equation}
Then, by (\ref{l6}) and upper bound estimates in Proposition \ref{upper}, we have
\begin{equation}
\begin{aligned}
\mu(E_{\varepsilon, \delta^*}) & \leq \mu\left(\left\{F\left(u_{\varepsilon}\right) \geq \delta_{1}\right\}\right) \\
& \leq \delta_1 \int_{\left\{F\left(u_{\varepsilon}\right) \geq \delta_{1}\right\}} F\left(u_{\varepsilon}\right) dx\\
& \leq C \varepsilon^{2} \mathbf{E}_\varepsilon \left(u_{\varepsilon}\right) \leq C \varepsilon.
\end{aligned}\label{l7}
\end{equation}
Since
\begin{equation}
  1=\mu(E_{\varepsilon, \delta^*})+\mu(\Omega_{\varepsilon, \delta}^+)+\mu(\Omega_{\varepsilon, \delta}^-), 
\end{equation}
we can only prove (\ref{l5}).
By $\rho\in Lip(\mathbb{R}^k)$ and  properties (\ref{H1}), there exists $C_0\geq 0$ such that
\begin{equation}
  |\rho(u_\varepsilon)|\leq 2Lip(\rho)|u_\varepsilon|, \quad F(u_\varepsilon)\geq \frac{c_4}{2}|u_\varepsilon|\text{ for } |u_\varepsilon|\geq C_0.
\end{equation}
This, combined with (\ref{l7}) and the upper bound estimate Proposition \ref{upper}, implies
\begin{equation}
\begin{aligned}
   &\left|\int_{E_{\varepsilon, \delta^*}}\rho(u_\varepsilon)dx\right|\\
  \leq &\left\| \rho(\cdot)\right\|_{L^{\infty}(B(C_0))}\mu\left(E_{\varepsilon, \delta^*}\right)+2L\int_{\{|u_\varepsilon|\geq
  C_0\}\cap E_{\varepsilon, \delta^*} }|u_\varepsilon|dx\\
  \leq &C\varepsilon+\frac{4L}{c_4}\int_{\Omega}F(u_\varepsilon)dx\\
  \leq &C\varepsilon+ \varepsilon^{2} \mathbf{E}_\varepsilon\leq
  C \varepsilon.
\label{l9}
\end{aligned}
\end{equation}
Then, using the fact $ \rho\left(\Pi\left(u_{\varepsilon}(x)\right)\right)=0$ in $\Omega_{\varepsilon, \delta^*}^-$, $\rho\in Lip(\mathbb{R}^k)$ and properties (\ref{H1}), we have
\begin{equation}
\begin{aligned}
 & \left|\int_{\Omega_{\varepsilon, \delta^*}^-}\rho(u_\varepsilon)dx\right|\\
 \leq &\int_{\Omega_{\varepsilon, \delta^*}^-}\left|\rho(u_\varepsilon)-\rho(\Pi (u_\varepsilon))\right|dx\\
\leq & Lip(\rho)\int_{\Omega_{\varepsilon, \delta^*}^-}\left|u_\varepsilon-\Pi(u_\varepsilon)\right|dx\\
\leq & C\left(\int_\Omega F(u_\varepsilon)dx\right)^{1/2}\leq C\varepsilon^{\frac{1}{2}}, 
\end{aligned}\label{l10}
\end{equation}
in the last step, we have used the Cauchy-Schwarz inequality and the upper bound estimate Proposition \ref{upper}.

Similar to the above argument, we have 
\begin{equation}
 \left|\int_{\Omega_{\varepsilon, \delta^*}^+}
\rho(u_\varepsilon)-\rho(\Pi (u_\varepsilon))dx\right|\leq C\varepsilon^{\frac{1}{2}}.\label{l11}
\end{equation}
We define
\begin{equation}
 m_\varepsilon:=\frac{\int_{\Omega_{\varepsilon, \delta^*}^+}\rho(\Pi (u_\varepsilon))dx}{\mu(\Omega_{\varepsilon, \delta^*}^+)}, \label{l12}
\end{equation}
it follows from $\rho(\Pi (u_\varepsilon))\in[m_1, m_2]$ in $\Omega_{\varepsilon, \delta}^+$ that
\begin{equation}
 m_\varepsilon\in[m_1, m_2].\label{l13}
\end{equation}
By combining (\ref{l9}), (\ref{l10}), and (\ref{l11}) with (\ref{l12}) and using the mass constraint $\int_{\Omega}\rho(u)dx=m$, we obtain
\begin{equation}
 |m_\varepsilon\mu(\Omega_{\varepsilon, \delta^*}^+)-m|\leq C\varepsilon^{\frac{1}{2}}.\label{l14}
\end{equation}
This implies that there exists sufficiently small $\varepsilon_0$, and $K_\varepsilon$ uniformly bounded for $\varepsilon<
\varepsilon_0$ such that
\begin{equation}
 m_\varepsilon\left(\mu(\Omega_{\varepsilon, \delta^*}^+)+K_\varepsilon\varepsilon^{\frac{1}{2}}\right)=m.\label{l15}
\end{equation}
By (\ref{l13}) and (\ref{l15}) we deduce
\begin{equation}
\mu(\Omega_{\varepsilon, \delta^*}^+)+K_\varepsilon\varepsilon^{\frac{1}{2}}\in[\frac{m}{m_2}, \frac{m}{m_1}], \label{l16}
\end{equation}
which completes the proof.
\end{proof}

\setcounter{claim}{0} Let us prove Proposition \ref{lower} and \ref{fine lower} by an argument used in \cite{andre2011minimization}.
\subsubsection{Rough lower bound estimates of \texorpdfstring{$\mathbf{E}_\varepsilon $}{} }
\begin{proposition}
 Under the same assumptions of Theorem \ref{theorem 1}. There exist a constant $C>0 $ and $\varepsilon_0>0$ such that when $0<\varepsilon< \varepsilon_0$, we have
\begin{equation}
\mathbf{E}_{\varepsilon} \geq \frac{ c_0^FI_\Omega\left({\sigma_m}_{}\right)}{\varepsilon}-\frac{C_{}}{\varepsilon^{\frac{1}{2}}}, \label{3.02}
\end{equation}

where $c_0^F$ and $I_\Omega\left({\sigma_m}_{}\right)$ are as defined in Theorem \ref{theorem 1}
\label{lower}
\end{proposition}
\begin{proof}

Using Proposition \ref{proposition 3.1} and the definition of ${\sigma_m}$ (\ref{1.7}), we have
$$
I_{\Omega}\left(\mu(\Omega_{\varepsilon, \delta^*}^+)+K_\varepsilon\varepsilon^{\frac{1}{2}}\right)\geq I_{\Omega}({\sigma_m}).
$$
Since $I_\Omega$ is locally Lipschitz on $(0, 1)$, it follows that
\begin{equation}
 I_{\Omega}\left(\mu(\Omega_{\varepsilon, \delta^*}^+)\right)\geq I_{\Omega}({\sigma_m})-C\varepsilon^{\frac{1}{2}}. \label{3.51}
\end{equation}
From properties (\ref{H1}), for $\delta\in(0, \delta_N)$, we have 
\begin{equation}
E_{\varepsilon, \delta} \subset\left\{F(u_\varepsilon) \geq c_1\delta^2\right\}, \label{l8}
\end{equation}
and from Proposition \ref{fine upper}, $\{\varepsilon \mathbf{E}_\varepsilon \}$ is bounded, that is
\begin{equation}
 \varepsilon \mathbf{E}_ \varepsilon \left(u_{\varepsilon}\right) \leq C, \label{3.412}
\end{equation}
Combining the above, by direct calculation, we obtain
$$
 c_1\varepsilon^{\frac{1}{2}} \mu\left(\left\{F\left(u_{\varepsilon}\right) \geq c_1\varepsilon^{\frac{1}{2}}\right\}\right) \leq \int_{\left\{F\left(u_{\varepsilon}\right) \geq c_1\varepsilon^{\frac{1}{2}}\right\}} F\left(u_{\varepsilon}\right) \leq \varepsilon^{2} \mathbf{E}_\varepsilon \left(u_{\varepsilon}\right) \leq C \varepsilon, \label{3.42}
$$
which implies 
\begin{equation}
  \mu\left\{ F\left(u_{\varepsilon}\right) \geq c_1\varepsilon^{\frac{1}{2}}
\right\}\leq C\varepsilon^{\frac{1}{2}}.
\end{equation}
Using this and (\ref{l8}), for any $a \in\left[\varepsilon^{\frac{1}{4}}, \frac{d_N}{2}\right]$, we obtain
\begin{equation}
 \begin{aligned}
&\left|\mu(\Omega_{\varepsilon, \delta^*}^+)-\mu(\Omega_{\varepsilon, a}^+)\right|\leq \mu\left(E_{\varepsilon, \varepsilon^{\frac{1}{4}}}\right)\leq \mu\left\{ F\left(u_{\varepsilon}\right) \geq c_1\varepsilon^{\frac{1}{2}}
\right\}\leq C\varepsilon^{\frac{1}{2}}.\label{3.55}
\end{aligned} 
\end{equation}
Because $I$ is locally Lipschitz on $(0, 1)$, using (\ref{3.55}) and (\ref{3.51}), we derive
\begin{equation*}
\begin{aligned}
 H^{n-1}(\partial^* \Omega_{\varepsilon, a}^+ \cap \Omega) 
&\geq I_{\Omega}\left(\mu(\Omega_{\varepsilon, a}^+)\right) \\
&\geq I_{\Omega}\left(\mu(\Omega_{\varepsilon, \delta^*}^+)\right)-C \varepsilon^{\frac{1}{2}} \\
&\geq I_{\Omega}({\sigma_m})-C\varepsilon^{\frac{1}{2}}.\label{3.2.56}
\end{aligned}
\end{equation*}
Note that $I_{\Omega}({\sigma})=I_{\Omega}({1-\sigma})$ for $\sigma\in[0, 1]$, similar to the above argument, we obtain
\begin{equation}
   H^{n-1}(\partial^* \Omega_{\varepsilon, a}^- \cap \Omega)\geq I_{\Omega}({\sigma_m})-C\varepsilon^{\frac{1}{2}}.
\end{equation}
Finally, using the co-area formula, we obtain
\begin{equation}
  \begin{aligned}
    \mathbf{E}_\varepsilon \left(u_{\varepsilon}\right)&\geq\frac{2}{\varepsilon} \int_{\varepsilon^{\frac{1}{4}}}^{\frac{d_{N}}{2}} \left( H^{n-1}(\partial^* \Omega_{\varepsilon, \lambda}^+ \cap \Omega) + H^{n-1}(\partial^* \Omega_{\varepsilon, \lambda}^- \cap \Omega) \right) \sqrt{f\left(\lambda^{2}\right)} d \lambda\\
    &\geq \frac{4(I_{\Omega}({\sigma_m})-C\varepsilon^{\frac{1}{2}})}{\varepsilon}\int_{\varepsilon^{\frac{1}{4}}}^{\frac{d_{N}}{2}}\sqrt{f\left(\lambda^{2}\right)} d \lambda\\
    &\geq \frac{4(I_{\Omega}({\sigma_m})-C\varepsilon^{\frac{1}{2}})}{\varepsilon}\left(\frac{c_0^F}{4}-\int_{0}^{\varepsilon^{\frac{1}{4}}}\sqrt{f\left(\lambda^{2}\right)} d \lambda\right)\\
    &\geq \frac{4(I_{\Omega}({\sigma_m})-C\varepsilon^{\frac{1}{2}})}{\varepsilon}\left(\frac{c_0^F}{4}-C\varepsilon^{\frac{1}{2}}\right)\\
    &\geq \frac{ c_0^FI_\Omega\left({\sigma_m}_{}\right)}{\varepsilon}-\frac{C_{}}{\varepsilon^{\frac{1}{2}}}, 
  \end{aligned}\label{3.2.57}
\end{equation}
where we have used the following estimate from properties (\ref{H1}): 
\begin{equation}
  \int_{0}^{\varepsilon^{\frac{1}{4}}}\sqrt{f\left(\lambda^{2}\right)} d \lambda\leq C\varepsilon^{\frac{1}{2}}.
\end{equation}

 This completes the proof of Proposition \ref{lower}.
\end{proof}

\subsubsection{Refined lower bound estimates of \texorpdfstring{$\mathbf{E}_\varepsilon $}{} }
Under the condition (\ref{G1}) for $I_\Omega$ at $\sigma_m$, we derive a more refined lower bound estimate for $\mathbf{E}_\varepsilon$.
\begin{proposition}
 Under the same assumptions of Theorem \ref{theorem 2}, there exists a constant $C>0 $ and $\varepsilon_0>0$ such that for $0<\varepsilon< \varepsilon_0$, we have
\begin{equation}
\mathbf{E}_{\varepsilon} \geq \frac{ c_0^FI_{\Omega}\left({\sigma_m}_{}\right)}{\varepsilon}-C, \label{3.021}
\end{equation}
where $c_0^F$ and $I_\Omega\left({\sigma_m}_{}\right)$ are defined as in Theorem \ref{theorem 2}

\label{fine lower}
\end{proposition}

\begin{proof}
 Similar to the argument for Proposition \ref{lower}. For any $a \in\left[L\varepsilon^{\frac{1}{2}}, \frac{d_N}{2}\right]$, we have
\begin{equation}
 \begin{aligned}
&\left|\mu(\Omega_{\varepsilon, \delta^*}^+)-\mu(\Omega_{\varepsilon, a}^+)\right|\leq \mu\left(E_{\varepsilon, L\varepsilon^{\frac{1}{2}}}\right)\leq \mu\left\{ F\geq c_1L^2\varepsilon
\right\}\leq \frac{C}{L^2}.
\end{aligned} 
\end{equation}
This and Proposition \ref{proposition 3.1}, for sufficiently large $L$, imply
\begin{equation}
 \mu(\Omega_{\varepsilon, a}^+)\in [\frac{m}{m_2}-\eta, \frac{m}{m_1}+\eta], 
\end{equation}
where $\eta$ is defined in the condition (\ref{G1}). Then, we use the condition (\ref{G1}) to get, 
\begin{equation*}
 H^{n-1}(\partial^* \Omega_{\varepsilon, a}^+ \cap \Omega) 
\geq I_{\Omega}\left( \mu(\Omega_{\varepsilon, a}^+)\right) 
\geq I_{\Omega}({\sigma_m}).
\end{equation*}
Notice $I_{\Omega}({\sigma})=I_{\Omega}({1-\sigma})$ for $\sigma\in[0, 1]$, similar to the above argument, we obtain
\begin{equation}
    H^{n-1}(\partial^* \Omega_{\varepsilon, a}^- \cap \Omega)\geq I_{\Omega}({\sigma_m}).
\end{equation}
Similar to the argument of (\ref{3.2.57}), we also have 
\begin{equation}
  \begin{aligned}
    \mathbf{E}_\varepsilon \left(u_{\varepsilon}\right)&\geq\frac{2}{\varepsilon} \int_{L\varepsilon^{\frac{1}{2}}}^{\frac{d_{N}}{2}} \left( H^{n-1}(\partial^* \Omega_{\varepsilon, \lambda}^+ \cap \Omega) + H^{n-1}(\partial^* \Omega_{\varepsilon, \lambda}^- \cap \Omega) \right) \sqrt{f\left(\lambda^{2}\right)} d \lambda\\
    &\geq \frac{4I_{\Omega}({\sigma_m})}{\varepsilon}\int_{L\varepsilon^{\frac{1}{2}}}^{\frac{d_{N}}{2}}\sqrt{f\left(\lambda^{2}\right)} d \lambda\\
    &\geq \frac{4I_{\Omega}({\sigma_m})}{\varepsilon}\left(\frac{c_0^F}{4}-\int_{0}^{L\varepsilon^{\frac{1}{2}}}\sqrt{f\left(\lambda^{2}\right)} d \lambda\right)\\
    &\geq \frac{4I_{\Omega}({\sigma_m})}{\varepsilon}\left(\frac{c_0^F}{4}-C(L)\varepsilon\right)\\
    &\geq \frac{ c_0^FI_\Omega\left({\sigma_m}_{}\right)}{\varepsilon}-C, 
  \end{aligned}\label{3.2.571}
\end{equation}
where we have used the following estimate from properties (\ref{H1}): 
\begin{equation}
  \int_{0}^{L\varepsilon^{\frac{1}{2}}}\sqrt{f\left(\lambda^{2}\right)} d \lambda\leq C(L)\varepsilon.\label{estimate 1}
\end{equation}
 This completes the proof of Proposition \ref{fine lower}.

\end{proof}

\begin{proof}[Proof of Theorem 1 and Theorem 2]
  Combing Proposition \ref{lower} and Proposition \ref{upper}, the proof of Theorem \ref{theorem 1} is complete.
Combing Proposition \ref{fine lower} and Proposition \ref{fine upper}, the proof of Theorem \ref{theorem 2}
 is complete.
\end{proof}

\section{Proof of Theorem \ref{Theorem 3}\label{section 4}}
This section is devoted to the proof of Theorem \ref{Theorem 3}. Recall from \cite{ambrosio1995new} that $SBV(\Omega)$, a subspace of $BV(\Omega)$ consists of all functions of bounded variations such that the Cantor part of the distributional derivatives is 0, that is 
\begin{equation}
  u\in SBV(\Omega)\iff Du=\nabla u dx+(u^+-u^-)H^{n-1}\llcorner J_u, 
\end{equation}
where $J_u $ denotes the set of jump discontinuity of $u$. For vector-valued functions, we define \begin{equation}
   SBV(\Omega, \mathbb{R}^k) = \{u=(u^1, \dots, u^k):u^i\in SBV(\Omega)\text{ for } i=1, \dots, k\}.
\end{equation}
We also recall the notion of $\Gamma$ from (\ref{gammer}). The proof of Theorem \ref{Theorem 3} consists of six claims.
\begin{proof}[Proof of Theorem 3]

\setcounter{claim}{0}

\begin{claim}
  There exists $\delta_\varepsilon\in[\frac{\delta^*}{2}, \delta^*],\varepsilon_0>0$ such that
  \begin{align}
    H^{n-1}(\Gamma)&\leq H^{n-1}(\partial^*\Omega^+_{\varepsilon, \delta_\varepsilon}\cap\Omega)\leq H^{n-1}(\Gamma)+C\varepsilon, \\
    H^{n-1}(\Gamma)&\leq H^{n-1}(\partial^*\Omega^-_{\varepsilon, \delta_\varepsilon}\cap\Omega)\leq H^{n-1}(\Gamma)+C\varepsilon,\label{bound}
  \end{align}
  for $\varepsilon<\varepsilon_0$.
  \label{theorem 3 claim 1}
\end{claim}
From the proof of Proposition \ref{proposition inf}, there exists $K>0,\varepsilon_0>0$ such that
\begin{equation}
 H^{n-1}(\partial^*\Omega^\pm_{\varepsilon, \delta}\cap\Omega) \geq H^{n-1}(\Gamma) \quad\forall \delta\in[K\varepsilon^{\frac{1}{2}}, \frac{d_N}{2}],\varepsilon<\varepsilon_0.\label{inf bound}
\end{equation}
Hence, we have from Proposition \ref{crude estimate} and Proposition \ref{fine upper}, 
\begin{equation}
  \begin{aligned}
   &\int_{\frac{\delta^*}{2}}^{\delta^*} \left( H^{n-1}(\partial^* \Omega_{\varepsilon, \lambda}^+ \cap \Omega) + H^{n-1}(\partial^* \Omega_{\varepsilon, \lambda}^- \cap \Omega) \right) \left( 2 \sqrt{f(\lambda^2)} \right) d \lambda\\
   \leq &c_0^F H^{n-1}(\Gamma) + C \varepsilon\\
&- 2 \left\{ \int_{0}^{\frac{\delta^*}{2}} + \int_{\delta^*}^{\frac{d_N}{2}} \right\} \left( H^{n-1}(\partial^* \Omega_{\varepsilon, \lambda}^+ \cap \Omega) + H^{n-1}(\partial^* \Omega_{\varepsilon, \lambda}^- \cap \Omega) \right) \sqrt{f(\lambda^2)} \, d \lambda\\
  \leq& c_0^F H^{n-1}(\Gamma) + C \varepsilon - 4H^{n-1}(\Gamma)  \left[ \left\{ \int_{K\varepsilon^{\frac{1}{2}}}^{\frac{\delta^*}{2}} + \int_{\delta^*}^{\frac{d_N}{2}} \right\} \sqrt{f(\lambda^2)} \, d\lambda \right]\\
  \leq& 4 H^{n-1}(\Gamma) \int_{\frac{\delta^*}{2}}^{\delta^*} \sqrt{f(\lambda^2)} \, d\lambda + C \varepsilon, 
  \end{aligned}\label{upper bound}
\end{equation}
where we have used the estimate (\ref{estimate 1}) in the last step.

This, combined with Fubini’s theorem, implies that there exists $\delta_\varepsilon\in(\frac{\delta^*}{2}, \delta^*)$ such that 
\begin{equation}
  H^{n-1}(\partial^* \Omega_{\varepsilon, \delta_\varepsilon}^+ \cap \Omega) + H^{n-1}(\partial^* \Omega_{\varepsilon,\delta_\varepsilon}^- \cap \Omega)\leq 2 H^{n-1}(\Gamma)+C\frac{\varepsilon}{{\delta^*}^2} .
\end{equation}
Combining (\ref{inf bound}) and (\ref{upper bound}), we obtain (\ref{bound}).
\begin{claim}
  For a subsequence $\varepsilon_i\to 0$, set $\delta_{\varepsilon_i}=\delta_i$, we have
  \begin{equation}
    \chi_{\Omega^\pm_{\varepsilon_i, \delta_i}}\rightharpoonup\chi_{\Omega^\pm} \text{ weakly in } BV(\Omega),
  \end{equation}
 where  $\Omega^\pm$ realizes the minimum $I(\sigma_m)$.
  \label{theorem 3 claim 2}
\end{claim}
From Claim \ref{theorem 3 claim 1} the family of functions $\{\chi_{\Omega^\pm_{\varepsilon_i, \delta_i}}\}$ is
 bounded in $BV(\Omega)$. Therefore, for a subsequence $\varepsilon_i\to 0$, there exists two sets $E^\pm\subset\Omega$ with finite perimeters such that
 \begin{equation}
    \chi_{\Omega^\pm_{\varepsilon_i, \delta_i}}\rightharpoonup \chi_{E^\pm} \text{ weakly in } BV(\Omega).
  \end{equation}
 By the lower semi-continuity of the $BV$-norm, 
 \begin{equation}
   H^{n-1}(\partial^*{E^\pm}\cap\Omega)\leq \liminf_{i\to\infty} H^{n-1}(\partial^*\Omega^\pm_{\varepsilon_i, \delta_i}\cap\Omega)=H^{n-1}(\Gamma), 
 \end{equation}
 with $  \mu(E^+)\in[\frac{m}{m_2}, \frac{m}{m_1}]$, combined with (\ref{gammer}), gives
 \begin{equation}
   \mu(E^\pm)=\mu(\Omega^\pm), 
 \end{equation}
 and $E^\pm=\Omega^\pm$ realizes the minimum $I(\sigma_m)$, where the same notions as in (\ref{gammer}) are used .
 \begin{claim}
 The following bound holds:
 \begin{equation}
     \int_{\Omega_{\varepsilon_i, \delta_i}^{+} \cup \Omega_{\varepsilon_i, \delta_i}^{-}}\left|\nabla\left(\Pi\left(u_{\varepsilon_i}\right)\right)\right|^{2} d x\leq C.
   \end{equation}
\label{theorem 3 claim 3}
\end{claim}
From (\ref{inf bound}) in Claim \ref{theorem 3 claim 1}, we have
\begin{equation}
 H^{n-1}(\partial^*\Omega^\pm_{\varepsilon_i, \delta}\cap\Omega) \geq H^{n-1}(\Gamma) \quad\forall \delta\in[K\varepsilon_i^{\frac{1}{2}}, \frac{d_N}{2}].
\end{equation}
Using this and the estimate (\ref{estimate 1}), we have
\begin{equation}
  \begin{aligned}
    & \frac{2}{\varepsilon_i}\int_{0}^{\frac{d_{N}}{2}} \left( H^{n-1}(\partial^* \Omega_{\varepsilon_i, \lambda}^+ \cap \Omega) + H^{n-1}(\partial^* \Omega_{\varepsilon_i, \lambda}^- \cap \Omega) \right)\sqrt{f\left(\lambda^{2}\right)} d \lambda\\
   \geq&\frac{4}{\varepsilon_i} \int_{K\varepsilon_i^{\frac{1}{2}}}^{\frac{d_{N}}{2}}H^{n-1}(\Gamma)\sqrt{f\left(\lambda^{2}\right)} d \lambda\\
  \geq&\frac{ c_0^F H^{n-1}(\Gamma)}{\varepsilon_i}-\frac{4H^{n-1}(\Gamma)}{\varepsilon_i}\int_{0}^{K\varepsilon_i^{\frac{1}{2}}}\sqrt{f\left(\lambda^{2}\right)} d \lambda\\
   \geq&\frac{c_0^F H^{n-1}(\Gamma)}{\varepsilon_i}-C, 
  \end{aligned}
\end{equation}
This, combined with the crude estimate in Proposition \ref{crude estimate} and the upper bound estimates in Proposition \ref{fine upper}, implies this claim.
\begin{claim}
  Define
  \begin{equation}
    v_i \equiv \Pi(u_{\varepsilon_i}) \chi_{\Omega_{\varepsilon_i, \delta_i}^+ \cup \Omega_{\varepsilon_i, \delta_i}^-}, 
  \end{equation}
  then there exists $v\in SBV(\Omega, \mathbb{R}^k)$ such that, after taking possible subsequences, 
  \begin{equation}
     v_i \to v \text{ in } L^1(\Omega, \mathbb{R}^k) \text{ and } \nabla v_i \to \nabla v \text{ weakly in } L^1(\Omega, \mathbb{R}^{nk}). \label{claim 4}
  \end{equation}\label{theorem 3 claim 4}
  \end{claim}
  It follows from the definition of $v_i$ that the absolute continuous part of its distributional derivative is given by
  \begin{equation}
    \nabla v_i = \nabla (\Pi(u_{\varepsilon_i})) \chi_{\Omega_{\varepsilon_i, \delta_i}^+ \cup \Omega_{\varepsilon_i, \delta_i}^-},
  \end{equation}
  then, by Claim \ref{theorem 3 claim 3}, 
  \begin{equation}
    \int_{\Omega} |\nabla v_i|^2 \, dx \leq \int_{\Omega_{\varepsilon_i, \delta_i}^+ \cup \Omega_{\varepsilon_i, \delta_i}^-} |\nabla (\Pi(u_{\varepsilon_i}))|^2 \, dx \leq C.\label{claim 4, 1}
  \end{equation}
  The set of jump discontinuities of $v_i$, denoted as $J_{v_i}$, satisfies
  \begin{equation}
    H^{n-1}(J_{v_i})=H^{n-1}(\partial^* \Omega_{\varepsilon_i, \delta_i}^+ \cap \Omega) + H^{n-1}(\partial^* \Omega_{\varepsilon_i, \delta_i}^- \cap \Omega)\leq 2 H^{n-1}(\Gamma)+C\varepsilon.\label{claim 4, 2}
  \end{equation}
 Additionally,
  \begin{equation}
    \|v_i\|_{L^\infty(\Omega)}\leq C.\label{claim 4, 3}
  \end{equation}
   It follows from (\ref{claim 4, 1}), (\ref{claim 4, 2}) and (\ref{claim 4, 3}) that $\{v_i\}$ is a
 compact sequence in $SBV(\Omega, \mathbb{R}^k)$ (see (\cite{ambrosio1995new}, p.128)). Hence, we may assume that there exists $v_i\in SBV(\Omega, \mathbb{R}^k)$ such that (\ref{claim 4}) holds, and
 \begin{equation}
   H^{n-1}\left(J_{v}\right)\leq \liminf _{i \to \infty} H^{n-1}\left(J_{v_{i}} \right).\label{claim 4, 4}
 \end{equation}
\begin{claim}
   The limiting function $v$ satisfies
  $$v \in H^{1}\left(\Omega^{ \pm}, N^{ \pm}\right)$$
   \label{theorem 3 claim 5}
\end{claim}
  By Fatou's Lemma, we have
\begin{equation}
  \int_{\Omega^\pm} |v - \Pi(v)| \, dx 
  \leq \liminf_{i\to\infty}\int_{\Omega_{\varepsilon_i, \delta_i}^\pm}|v-v_i|dx\to 0, 
\end{equation}
which implies $v\in L^1(\Omega^\pm, N^\pm)$.
From Claim \ref{theorem 3 claim 4}, we conclude
$$
H^{n-1}\llcorner\left(\partial^* \Omega_{\varepsilon_{i}, \delta_i}^{ \pm} \cap \Omega\right) \rightharpoonup H^{n-1}\llcorner\Gamma
$$
as weak convergence of Radon measures. By the lower semi-continuity of Radon measures, we have for any $\eta>0$,
$$
\begin{aligned}
H^{n-1}(\Gamma) =&H^{n-1}\left(\Gamma \cap\left\{x \in \Omega: {d}_{\Gamma}(x)\leq\eta\right\}\right) \\
\leq& \liminf _{i \to \infty}H^{n-1}\left(\left\{x \in \Omega: {d}_{\Gamma}(x)\leq\eta\right\}\cap \left(\partial^* \Omega_{\varepsilon_{i}, \delta_i}^{ \pm} \cap \Omega\right) \right)\\
\leq&\lim_{i\to\infty} H^{n-1}\left(\partial^* \Omega_{\varepsilon_{i}, \delta_i}^{ \pm} \cap \Omega\right)\\
&-\limsup _{i \to \infty}H^{n-1}\left(\left\{x \in \Omega: {d}_{\Gamma}(x)>\eta\right\}\cap \left(\partial^* \Omega_{\varepsilon_{i}, \delta_i}^{ \pm} \cap \Omega\right) \right)\\
\leq& H^{n-1}(\Gamma)-\limsup _{i \to \infty}H^{n-1}\left(\left\{x \in \Omega: {d}_{\Gamma}(x)>\eta\right\}\cap \left(\partial^* \Omega_{\varepsilon_{i}, \delta_i}^{ \pm} \cap \Omega\right) \right).
\end{aligned}
$$
This implies
\begin{equation}
  \limsup _{i \to \infty}H^{n-1}\left(\left\{x \in \Omega: {d}_{\Gamma}(x)>\eta\right\}\cap \left(\partial^* \Omega_{\varepsilon_{i}, \delta_i}^{ \pm} \cap \Omega\right) \right)=0.
\end{equation}
 From (\ref{claim 4, 4}) and lower semi-continuity, we have
$$
\begin{aligned}
& H^{n-1}\left(J_{v} \cap\left\{x \in \Omega^{ \pm}: {d}_{\Gamma}(x)>\eta\right\}\right) \\
& \quad \leq \liminf _{i \to \infty} H^{n-1}\left(J_{v_{i}} \cap\left\{x \in \Omega^{ \pm}: {d}_{\Gamma}(x)>\eta\right\}\right) \\
& \quad \leq \liminf _{i \to \infty} H^{n-1}\left((\partial^* \Omega_{\varepsilon_{i}, \delta_i}^+\cup\partial^* \Omega_{\varepsilon_{i}, \delta_i}^-)\cap\Omega\cap\left\{x \in \Omega: {d}_{\Gamma}(x) >\eta\right\}\right)=0 .
\end{aligned}
$$
This implies
\begin{equation*}
H^{n-1}\left(J_{v} \cap \Omega^{ \pm}\right)=0, 
\end{equation*}
thus, $v \in H^{1}\left(\Omega^{ \pm}, N^{ \pm}\right)$. 

\begin{claim}
 The sequence $ u_{\varepsilon_i}\to v \text{ in } L^1(\Omega, \mathbb{R}^k)$. Specifically, we have 
 \begin{equation}
   \int_\Omega\rho(v)dx=m.
 \end{equation}\label{theorem 3 claim 6}
 \end{claim}
On the one hand, 
\begin{equation}
  \int_{\Omega_{\varepsilon_i, \delta_i}^\pm} |u_{\varepsilon_i} - \Pi(u_{\varepsilon_i})| \, dx 
  \leq C \int_{\Omega_{\varepsilon_i, \delta_i}^\pm} \sqrt{F(u_{\varepsilon_i})} \, dx \leq C \sqrt{\varepsilon_i} \to 0.
\end{equation}
On the other hand, as in (\ref{l9}), we have
\begin{equation}
\begin{aligned}
   &\int_{E_{\varepsilon_i, \delta_i}}|u_{\varepsilon_i}|dx\\
  \leq &C_0\mu\left(E_{\varepsilon_i, \delta_i}\right)+\int_{\{|u_{\varepsilon_i}|\geq
  C_0\}\cap E_{\varepsilon_i, \delta_i} }|u_{\varepsilon_i}|dx\\
  \leq &C\varepsilon_i+\frac{2}{c_4}\int_{\Omega}F(u_{\varepsilon_i})dx\\
  \leq &C\varepsilon_i\to 0.
\label{4.24}
\end{aligned}
\end{equation}
These two inequalities, combined with
\begin{equation}
  \Pi(u_{\varepsilon_i}) \chi_{\Omega_{\varepsilon_i, \delta_i}^+ \cup \Omega_{\varepsilon_i, \delta_i}^-} \to v \quad \text{in } L^1(\Omega), 
\end{equation}
implies $u_{\varepsilon_i}\to v$ in $L^1(\Omega)$. In particular, 
\begin{equation}
  \begin{aligned}
     &\left|\int_\Omega\rho(u_{\varepsilon_i})dx-\int_\Omega\rho(v)dx\right|\\
    \leq&\int_\Omega|\rho(u_{\varepsilon_i})-\rho(v)|dx\\  
    \leq&Lip(\rho)\int_\Omega|u_{\varepsilon_i}-v|dx\to 0, 
  \end{aligned}
\end{equation}
which implies $\int_\Omega\rho(v)dx=m$. 

So far, we have proved all the conclusions stated in Theorem \ref{Theorem 3}.
\end{proof}

\section{ The results for the type II mass constraint
 \label{section 5}}
In the section, we summarize all the results for the type II mass constraint. Assume
\begin{equation}
\widetilde{F}(p)=\widetilde{f}\left(d_0(p, N)\right), \label{1.1.4011}
\end{equation}
where $d_0(\cdot, N)$ is defined in (\ref{d_0 def}). Following the settings described in \cite{leoni2016second}, we choose $\widetilde{f}$ with the following properties:
\begin{equation}
 \begin{aligned}
&1.\text{ }\widetilde{f} \text{ is of class } C^{2}(\mathbb{R} \backslash\{-\frac{d_N}{2}, \frac{d_N}{2}\}) \text{ and has precisely two zeros at } \pm\frac{d_N}{2}, \\
&2. \text{ }\lim _{s \to -\frac{d_N}{2}} \frac{\widetilde{f}^{\prime \prime}(s)}{|s+\frac{d_N}{2}|^{q-1}}=\lim _{s \to \frac{d_N}{2}} \frac{\widetilde{f}^{\prime \prime}(s)}{|s-\frac{d_N}{2}|^{q-1}}:=\ell>0, \quad q \in(0, 1), \\
&3. \text{ }\widetilde{f}^{\prime} \text { has exactly } 3 \text { zeros at } -\frac{d_N}{2}<c<\frac{d_N}{2}, \quad \widetilde{f}^{\prime \prime}(c)<0, \\
& 4.\text{ }\liminf _{|s| \to \infty}\left|\widetilde{f}^{\prime}(s)\right|>0.
\end{aligned}\label{wide f}
\end{equation} 
Here we only consider $q\in(0, 1)$, as in the case, the support set of the optimal profile of the phase transition is bounded, which plays an essential role in characterizing the expansion of the below $\widetilde{ \mathbf{E}}_\varepsilon$. Define $z$ as the unique global solution to the Cauchy problem
\begin{equation}
 \begin{cases}z^{\prime}(t)=\sqrt{\widetilde{f}(z(t))} & \text { for } t \in \mathbb{R}, \\ z(0)=c, & z(t) \in[-\frac{d_N}{2}, \frac{d_N}{2}].\end{cases}\label{z}
\end{equation}
The function $z$ is called the \textit{optimal profile} of the phase transition. By integration, 
\begin{equation}
 t=\int_0^{z(t)}\frac{ds}{\sqrt{\widetilde{f}(s)}}, 
\end{equation}
which is finite thanks to (\ref{wide f}), since $q<1$. It follows that there exist $t_1, t_2\in \mathbb{R}$ such
$z(t)=\frac{d_N}{2}$ for $t\geq t_2$ and $z(t)=-\frac{d_N}{2}$ for $t\leq t_1$.

We consider the mass constraint
\begin{equation}
\int_{\Omega}d_0(u, N)dx=m, \label{1.12}
\end{equation}
where $-\frac{d_N}{2}< m<\frac{d_N}{2}$.

Let $u\in H^1(\Omega)$, define the energy functional 
\begin{equation*}
\widetilde{\mathbf{E}}_{\varepsilon}(u):= 
\int_{\Omega}\left(|\nabla u|^{2}+\frac{1}{\varepsilon^{2}}\widetilde{ F}(u)\right) d x, 
\end{equation*}
and
\begin{equation}
\widetilde{\mathbf{E}}_{\varepsilon}:=\inf_{u\in H^1(\Omega)} \left\{
\widetilde{\mathbf{E}}_{\varepsilon}(u):\int_{\Omega}d_0(u, N)dx=m
\right\}, \label{1.90}
\end{equation}
where $m\in (-\frac{d_N}{2}, \frac{d_N}{2})$.

Similar to \cite{leoni2016second}, to characterize $\widetilde{\mathbf{E}}_{\varepsilon} $, we also assume $I_{\Omega}(\sigma)$ admits a Taylor expansion of order $2$ at $\sigma_m$:
\begin{equation}
I_{\Omega}(\sigma)=I_{\Omega}\left(\sigma_m\right)+I_\Omega^{\prime}\left(\sigma_m\right)\left(\sigma-\sigma_m\right)+O\left(\left|\sigma-\sigma_m\right|^{1+\beta}\right),\label{taylor}\tag{$\widetilde{G}$}
\end{equation}
where $\sigma_m=\frac{\frac{d_N}{2}-m}{d_N}$, and $\beta \in(0, 1]$. For domains of class $C^2$, $I_{\Omega}$ is semi-concave (see \cite{bavard1986volume}) and (\ref{taylor}) holds with $\beta= 1$ at $L^1$ a.e. $\sigma_m$ in $[0, 1]$ (see \cite{urruty1993convex}).
For $n$, we also restrict $n\leq7$ to guarantee regularity of the minimizers of the problem (\ref{minimizers}). 

Now, we state our results.
\begin{theorem}
Assume $n\leq7$, and $ I_\Omega$ satisfies (\ref{taylor}). Assume
\begin{equation}
 \widetilde{F}(p)=\widetilde{f}\left(d_0(p, N)\right), 
\end{equation}
and $\widetilde{\mathbf{E}}_{\varepsilon}$ is defined as in (\ref{1.9}). Then, for small $\varepsilon$, we have

\begin{equation}
 \widetilde{\mathbf{E}}_{\varepsilon} = \frac{c_0^{\widetilde{F}} I_\Omega({\sigma_m})}{\varepsilon}+\left(c_{\mathrm{sym}}+c_0^{\widetilde{F}} \tau_{0}\right)(n-1) \kappa_{0}I_\Omega({\sigma_m})+o(1).
\end{equation}

Here $\kappa_{0}$ is the constant mean curvature of the set $ \partial E\cap\Omega$, where $E$ minimizes (\ref{minimizers}) at $\sigma_m=\frac{\frac{d_N}{2}-m}{d_N}$ and $c_{\mathrm{sym}}$ is a constant defined by
\begin{equation*}
c_{\mathrm{sym}}:=2\int_{\mathbb{R}}\widetilde{ f}(z(t)) t d t. 
\end{equation*}

Here $\tau_{0}$ is a constant such that
\begin{equation*}
\int_{\mathbb{R}} z\left(t-\tau_{0}\right)-\operatorname{sgn}_{\pm\frac{d_N}{2}}(t) d t=0, 
\end{equation*}
where $z$ is defined in (\ref{z}); $\operatorname{sgn}_{\pm\frac{d_N}{2}}(t)=\frac{d_N}{2}$ for $t\geq0$ and $\operatorname{sgn}_{\pm\frac{d_N}{2}}(t)=-\frac{d_N}{2}$ for $t<0$.

Here $c_0^{\widetilde{F}}$ is a constant defined by
\begin{equation}
 c_0^{\widetilde{F}}:=2\int_{-\frac{d_N}{2}}^{\frac{d_N}{2}}\sqrt{\widetilde{f}(\lambda)}d\lambda.
\end{equation}
\label{theorem 4 }
\end{theorem}
\begin{remark}
 This result is consistent with the results for the potentials of one-dimensional wells in \cite{leoni2016second}.
\end{remark}
 We list some important properties of $d_0(\cdot, N)$ here (refer to (\ref{d_0 def}) for its definition). 
\begin{lemma}
 The function $d_0(\cdot, N):\mathbb{R}^k\to[\frac{d_N}{2}, \frac{d_N}{2}]$ is Lipschitz continuous on $\mathbb{R}^k$, and satisfies
 \begin{align}
 &1. |\nabla d_0(\cdot, N)|\leq 1 \quad \text{ a.e on } \mathbb{R}^k. \label{d_0 1}\\
 & 2. d_0(p, N) =\begin{cases}
 - \frac{d_N}{2}\quad &\text{ if and only if }\quad q\in N^-,\\
 \frac{d_N}{2}\quad &\text{ if and only if }\quad q\in N^+.\label{d_0 2}
 \end{cases}
 \end{align}\label{Lemma 5.1}
\end{lemma}
\begin{proof}
 
 1. If $d_0(p, N)<0$ and $d_0(q, N)> 0$. Since $N^\pm$ is compact, let $d_0(p, N)=d(p, p^\prime)- \frac{d_N}{2}$ and $d_0(q, N)=\frac{d_N}{2}-d(q, q^\prime)$, where $p^\prime\in N^-\text{ and }q^\prime\in N^+$. Then, from the definition, we get
 \begin{equation}
 \begin{aligned}
 0 \leq d_0(q, N)-d_0(p, N)&= d_N-(d(q, q\prime)+d(p, p\prime))\\
 &\leq d(p^\prime, q^\prime)-(d(q, q\prime)+d(p, p\prime))\\
 &\leq d(p, q).
 \end{aligned}
 \end{equation}
 If $d_0(p, N)<0$ and $d_0(q, N)=0$. A similar argument gives 
 \begin{equation}
 \begin{aligned}
 0\leq d_0(q, N)-d_0(p, N)&=\frac{d_N}{2}-d(p, N^-)\\
 & \leq d(q, N^-)-d(p, N^-)\\
 & \leq d(p, q).
 \end{aligned}
 \end{equation}
 Other cases can be treated in a similar way, ensuring that $d_0(\cdot, N)$ is Lipschitz continuous on $\mathbb{R}^k$ and (\ref{d_0 1}) holds.
 
 2. The formula (\ref{d_0 2}) follows directly from the definition (\ref{d_0 def}).
\end{proof}
\subsection{P\'olya-Szeg{\H o} type inequality}
To obtain the lower bound of $ \widetilde{\mathbf{E}}_{\varepsilon}$, we use the P\'olya-Szeg{\H o} inequality to reduce the problem to one dimension. According to the classical P\'olya-Szeg{\H o} inequality, for a positive function \(u \in W^{1, p}(\mathbb{R}^n) \), its decreasing spherical rearrangement \(u^{*} \) will not increase the \(L^p \) norm of the gradient in \(\mathbb{R}^n \). The techniques are identical to those used in \cite{leoni2016second}, we include all required propositions.

We first recall that the isoperimetric function $I_\Omega$ satisfies (see, \cite{maz2011sobolev,cianchi2008neumann, alberico2007borderline}) 
\begin{equation}
 I_\Omega(\sigma)\geq C_1 {\min\{\sigma, 1-\sigma\}}^{\frac{n-1}{n}}\quad\text{ for all } \sigma\in[0, 1].\label{propoey}
\end{equation}
We need the following proposition.
\begin{proposition}{(\cite{leoni2016second}, Proposition 3.1)}
 Suppose that $I_{\Omega}$ satisfies (\ref{taylor}). Then there exists a function $I_{\Omega}^{*} \in C_{\mathrm{loc}}^{1, \beta}(0, 1)$ such that
\begin{align}
I_{\Omega}^{*}(\sigma) & =I_{\Omega}^{*}(1-\sigma) \quad \text { for all } \sigma\in(0, 1), \label{identy} \\
I_{\Omega} & \geq I_{\Omega}^{*}>0 \quad \text { in }(0, 1), \\
I_{\Omega}\left(\sigma_m\right) & =I_{\Omega}^{*}\left(\sigma_m\right), \quad I_{\Omega}^{\prime}\left(\sigma_m\right)=\left(I_{\Omega}^{*}\right)^{\prime}\left(\sigma_m\right), \\
I_{\Omega}^{*}(\sigma) & =C_{0} \sigma^{\frac{n-1}{n}} \quad \text { for all } \sigma \in(0, \delta), \text{ for } C_{0}>0 \text{ and } 0<\delta<1.
\end{align}\label{proposition i*}
\end{proposition}

We define the function \(V_{\Omega} \) as the solution to the following Cauchy problem:
\begin{equation}
 \frac{{d}}{{d}t} V_{\Omega}(t) = I_{\Omega}^{*}(V_{\Omega}(t)), \quad V_{\Omega}(0) = \frac{1}{2}.
\end{equation}
From inequality (\ref{propoey}), there exists \(T > 0 \) such that \(0 < V_{\Omega}(t) \) for \(-T < t < 0 \), and moreover, \(V_{\Omega}(-T) = 0 \). Furthermore, from equation (\ref{identy}), we conclude that \(V_{\Omega}(T) = 1 \) and \(V_{\Omega}(t) < 1 \) for all \(0 < t < T \). Define
\begin{equation}
 I:=(-T, -T).\label{interal}
\end{equation}

For $y \in \mathbb{R}^{n}$, denote $y = (y', y_n) \in \mathbb{R}^{n-1} \times \mathbb{R}$. We define a rearranged version of $\Omega$ as 
\begin{equation}
 \Omega^{*}:= \left\{ y:y_n \in I, \, y' \in B_{n-1}(0, r(y_n)) \right\}, 
\end{equation}
where for each $t \in I$, 
\begin{equation}
 r(t):= \left(\frac{I_{\Omega}^{*}(V_{\Omega}(t))}{\alpha_{n-1}} \right)^{1/(n-1)}, \quad \alpha_{n-1}:= L^{n-1}\left(B_{n-1}(0, 1) \right), 
\end{equation}
and \(\alpha_{n-1} \) represents the \((n-1)\)-dimensional Lebesgue measure of the unit ball \(B_{n-1}(0, 1) \). 

Note that the definition of \(r(t) \) implies the relation
\begin{equation}
L^{n-1}\left(B_{n-1}(0, r(t)) \right) = I_{\Omega}^{*}(V_{\Omega}(t)) 
\end{equation}
holds for all $ t \in I $.

For any measurable function \(u: \Omega \to \mathbb{R} \), we define the distribution function \(\varrho_{u}(s):= L^{n}(\{ u > s \}) \), we introduce the following function:
\begin{equation}
 g_u(t):= \sup \left\{ s \in \mathbb{R}: \varrho_{u}(s) > V_{\Omega}(t) \right\}. 
\end{equation}
We also define the function \(P_u(t) \) as
\begin{equation}
P_u(t):= g_u(-t).
\end{equation}
Now we give the main result of this subsection. 
\begin{proposition}{(\cite{leoni2016second}, Corollary 3.12)}
 Let $u \in H^{1}(\Omega)$. Then the following inequality holds:
\begin{equation*}
\int_{\Omega}\left(|\nabla u|^{2} +\frac{1}{\varepsilon^{2}}\widetilde{f}(u)\right)d x \geq \int_{I}\left(\left(P_{u}^{\prime}\right)^{2}+ \frac{1}{\varepsilon^{2}}\widetilde{f}\left(P_{u}\right)\right) I_{\Omega}^{*}\left(V_{\Omega}\right) dt.
\end{equation*}
 Moreover, the integral constraint satisfies 
 $$
\int_{\Omega} u d x=\int_{I} P_u I_{\Omega}^{*}\left(V_{\Omega}\right) d t.
$$
\label{polya}
\end{proposition}

\subsection{An one-dimensional functional problem}

For $u:I\to \mathbb R\in H^1(\Omega)$, we define one-dimensional energy functional
\begin{equation*}
\widetilde{\mathbf{E}}_{\varepsilon}^1(u):= 
\int_{ I}\left(|\nabla u|^{2}+\frac{1}{\varepsilon^{2}} \widetilde{f}(u)\right) I_{\Omega}^{*}\left(V_{\Omega}\right)d t, 
\end{equation*}
$\widetilde{f}$ is defined as in (\ref{wide f}), $I$ is defined as in (\ref{interal}), $ I_{\Omega}^{*}$ is defined as in Proposition \ref{proposition i*}. We also define
\begin{equation}
\widetilde{\mathbf{E}}_{\varepsilon}^1:=\inf_{u\in H^1(I)} \left\{
\mathbf{E}_{\varepsilon}(u):\int_{I}uI_{\Omega}^{*}\left(V_{\Omega}\right)dt=m
\right\}, \label{1.99}
\end{equation}
where $m\in \left(-\frac{d_N}{2}\int_{I}I_{\Omega}^{*}\left(V_{\Omega}\right)dt, \frac{d_N}{2}\int_{I}I_{\Omega}^{*}\left(V_{\Omega}\right)dt\right)$. For $ \widetilde{\mathbf{E}}_\varepsilon^1 $, we have the following:
\begin{proposition}{(\cite{leoni2016second}, Theorem 4.17)}
Assume $n\leq7$ and $ I_\Omega^*$ satisfies (\ref{taylor}). 
$\widetilde{\mathbf{E}}_{\varepsilon}^1$ is defined in (\ref{1.9}). Then, there exists $\varepsilon_0>0$ such that for $0<\varepsilon\leq \varepsilon_0$, we have
\begin{equation}
 \widetilde{\mathbf{E}}_{\varepsilon}^1\geq\frac{c_0^{\widetilde{F}} I_{\Omega}^{*}(\sigma_m)}{\varepsilon}+\left(c_{\mathrm{sym}}+c_0^{\widetilde{F}} \tau_{0}\right) {(I_{\Omega}^{*})}^{\prime}(\sigma_m)+o(1),
\end{equation}\label{proposition inf}
Here $ c_{\mathrm{sym}}, \tau_0, c_0^{\widetilde{F}}, \sigma_m$ are defined as in Theorem \ref{theorem 4 }.
\end{proposition}

From Proposition \ref{proposition i*} and Lemma 5.4 in \cite{leoni2016second}, we know
\begin{equation}
\begin{aligned}
I_{\Omega}^{*}(\sigma_m)&=I_{\Omega}^{}(\sigma_m), \\
 {(I_{\Omega}^{*})}^{\prime}(\sigma_m)&={(I_{\Omega}^{})}^{\prime}(\sigma_m)=(n-1)\kappa_0I_{\Omega}(\sigma_m).
\end{aligned}\label{ineq}
\end{equation}
\subsection{Lower bound estimates of \texorpdfstring{$\widetilde{\mathbf{E}}_\varepsilon$}{}}
The estimates of the upper bound is based on the P\'olya-Szeg{\H o} inequality and Proposition \ref{proposition inf}. 

\begin{proposition}
 Under the same assumptions of Theorem \ref{theorem 4 }, there exists $\varepsilon_0>0$ such that for $0<\varepsilon\leq \varepsilon_0$, we have
\begin{equation}
\widetilde{\mathbf{E}}_\varepsilon\geq\frac{c_0^{\widetilde{F}} I_\Omega({\sigma_m})}{\varepsilon}+\left(c_{\mathrm{sym}}+c_0^{\widetilde{F}} \tau_{0}\right)(n-1) \kappa_{0}I_\Omega({\sigma_m})+o(1), 
\label{3.03111}
\end{equation}
where $\kappa_0, c_{\mathrm{sym}}, \tau_{0}, c_0^{\widetilde{F}}$ are as defined in Theorem \ref{theorem 4 }.
\label{fine lower 11}
\end{proposition}
\begin{proof}
Set $\{u_\varepsilon\}\subset H^ 1(\Omega)$, satisfying $\int_\Omega d_0(u_\varepsilon, N)dx=m$ and define 
\begin{equation}
 v_\varepsilon= d_0(u_\varepsilon, N).
 \end{equation}
Since $|\nabla d_0(q, N)|\leq1$ a.e. on $\mathbb{R}^k$ from Lemma \ref{Lemma 5.1} and $u_\varepsilon\in H^1(\Omega)$, we have $v_\varepsilon\in H^1(\Omega)$ and 
 \begin{equation}
 |\nabla u_\varepsilon(x)|\geq|\nabla v_\varepsilon(x)|\quad \text{ a.e. } x\in\Omega.
 \end{equation}
Then, we use $\widetilde{F}(u_\varepsilon)=\widetilde{f}(v_\varepsilon)$ and Proposition \ref{polya}, we have
 \begin{equation}
 \begin{aligned}
 & \int_{\Omega}\left(|\nabla u_\varepsilon|^{2}+\frac{1}{\varepsilon^{2}} \widetilde{F}(u_\varepsilon)\right) d x\\
 \geq & \int_{\Omega}\left(|\nabla v_\varepsilon|^{2}+\frac{1}{\varepsilon^{2}} \widetilde{f}\left(v_\varepsilon\right)\right) d x\\
 \geq & \int_{I}\left(\widetilde{f}\left(P_{v_\varepsilon}\right)+\varepsilon^{2}\left(P_{v_\varepsilon}^{\prime}\right)^{2}\right) I_{\Omega}^{*}\left(V_{\Omega}\right) dt,
 \end{aligned}
 \end{equation}

with
\begin{equation}
\int_{\Omega} v_\varepsilon d x=\int_{I} P_{ v_\varepsilon} I_{\Omega}^{*}\left(V_{\Omega}\right) d t.
\end{equation}

Thus, by the definition of $\widetilde{\mathbf{E}}_\varepsilon^1 $, we obtain
\begin{equation}
\widetilde{\mathbf{E}}_\varepsilon\geq\widetilde{\mathbf{E}}_\varepsilon^1.
\end{equation}

This, combined with (\ref{3.031}) and (\ref{ineq}), implies (\ref{3.03111}). The proof is complete.
\end{proof}

\subsection{Upper bound estimates of \texorpdfstring{$\widetilde{\mathbf{E}}_\varepsilon$}{}}
The upper bound estimate are based on the constructions, similar to that in (Theorem 4.7, \cite{leoni2016second}).
\begin{proposition}
 Under the same assumptions of Theorem \ref{theorem 4 }, there exists $\varepsilon_0>0$ such that for $0<\varepsilon\leq \varepsilon_0$, we have
\begin{equation}
\mathbf{E}_{\varepsilon} ^0\leq \frac{c_0^{\widetilde{F}} I_\Omega({\sigma_m})}{\varepsilon}+\left(c_{\mathrm{sym}}+c_0^{\widetilde{F}} \tau_{0}\right)(n-1) \kappa_{0}I_\Omega({\sigma_m})+o(1), 
\label{3.031}
\end{equation}
where $\kappa_0, c_{\mathrm{sym}}, \tau_{0}, c_0^{\widetilde{F}}$ are defined same as in Theorem \ref{theorem 4 }.
\label{fine upper 1}
\end{proposition}

\begin{proof}
 Define the minimizer $E$ of (\ref{minimizers}) at $\sigma_m$ and set
\begin{equation}
 \Gamma=\overline{\partial E\cap\Omega}.
\end{equation}
From \cite{gruter1987boundary} we know that $ \Gamma$ is a $C^{2, \alpha}$ surface, that intersects $\partial\Omega$ orthogonally. Define 
\begin{equation}
 \eta(t)= H^{n-1}\{x\in\Omega|{d}_\Gamma(x)=t\}.
\end{equation}
For above $\eta$, we can assume, for some $T>0$, 
\begin{equation}
 \eta(t)=0 \text{ for } |t|\geq T.
\end{equation}
Define $ I= [-T, T]$ and let $(p^+, p ^-)\in M^+ \times M^- $ such that $|p^+-p ^-| =d_N$. Now we construct comparison maps $v_\varepsilon\in H^1:\Omega\to \mathbb{R}^k$ 
\begin{equation*}
v_{\varepsilon}(x)=\frac{p_{}^{+}+p_{}^{-}}{2}+z_{\varepsilon}({d}_\Gamma(x)) \frac{p_{}^{+}-p_{}^{-}}{\left|p_{}^{+}-p_{}^{-}\right|}, \quad x \in \Omega, 
\end{equation*}
where
$$z_{\varepsilon}(t):=z\left(\frac{t}{\varepsilon}-\tau_\varepsilon\right), $$
where $\tau_\varepsilon$ is selected to satisfy (\ref{1.12}) and the function $z$ in (\ref{z}) satisfies $z(t) \equiv \frac{d_{N}}{2}$ for $t \geq t_{2}$ and $z(t) \equiv -\frac{d_{N}}{2}$ for $t \leq t_{1}$.
From the construction, the mass constraint is 
\begin{equation}
 m=\int_{\Omega}d_0(v_\varepsilon, N)dx=
 \int_\Omega z_{\varepsilon}({d}_\Gamma(x))dx.
 \label{up1}
\end{equation}
For the energy, we compute 
\begin{equation}
\int_{\Omega}\left(|\nabla v_\varepsilon|^{2}+\frac{1}{\varepsilon^{2}} F(v_\varepsilon)\right) d x
 =\int_\Omega\left(|z_{\varepsilon}^{\prime}({d}_\Gamma(x))|^2+\frac{1}{\varepsilon^{2}} \widetilde{f}(z_{\varepsilon}({d}_\Gamma(x)))\right)dx.
\label{up7}
\end{equation}

Using the same argument as in (\cite{leoni2016second}, Theorem 4.7, Step 2) to complete the proof.
\end{proof}
\begin{proof}[Proof of Theorem 4]
  Combing the lower and upper bound estimates from Proposition \ref{fine lower 11} and Proposition \ref{fine upper 1}, we obtain Theorem \ref{theorem 4 }.
\end{proof}
\printbibliography
\end{document}